\documentclass[a4paper,12pt]{amsart}

\usepackage{fullpage}
\usepackage{amssymb,amsmath}
\usepackage{float}
\usepackage{graphicx}
\usepackage{epsf}
\usepackage{ psfrag}
\usepackage{subcaption}
\usepackage{cleveref}
\usepackage[all]{xy}
\usepackage{caption}
\usepackage{enumerate,enumitem}
\usepackage{inputenc}
\usepackage[usenames, dvipsnames]{color}
\usepackage{tikz}
\usetikzlibrary[shapes]
\usepackage{multirow}
\usepackage{xparse}
\usepackage{lscape}

\newcommand\scalemath[2]{\scalebox{#1}{\mbox{\ensuremath{\displaystyle #2}}}}
\theoremstyle{definition}
\newtheorem{lemma}{Lemma}[section]
\newtheorem{prop}[lemma]{Proposition}
\newtheorem{proposition}[lemma]{Proposition}

\newtheorem{thm}[lemma]{Theorem}
\newtheorem{theorem}[lemma]{Theorem}

\newtheorem{de}[lemma]{Definition}
\newtheorem{example}[lemma]{Example}

\newtheorem{remark}[lemma]{Remark}

\newtheorem*{thm*}{Theorem}
\newtheorem*{cor*}{Corollary}

\newcommand{\CC}{{\mathbb {C}}}
\newcommand{\ZZ}{{\mathbb {Z}}}

\newcommand{\ch}{{\operatorname{ch}}}

\DeclareMathOperator{\Gr}{Gr}

\begin{document}

\title{Combinatorial model for $m$-cluster categories in type $E$}

\author[Bing Duan, Lisa Lamberti, Jian-Rong Li]{Bing Duan, Lisa Lamberti, Jian-Rong Li}

\address{Bing Duan, School of Mathematics and Statistics, Lanzhou University, Lanzhou, China}
\email{duan890818@163.com}
\address{Lisa Lamberti, Department of Biosystems Science and Engineering, ETH Z\"{u}rich, Basel, Switzerland; SIB Swiss Institute of Bioinformatics, Basel Switzerland}
\email{lisa.lamberti@bsse.ethz.ch}
\address{Jian-Rong Li, Institute of Mathematics and Scientific Computing, University of Graz, Graz, Austria}
\email{lijr07@gmail.com}

\begin{abstract}
In this paper, we revisit the geometric description of 
cluster categories in type $E$ in terms of 
colored diagonals in a polygon given in \cite{L3}.
We then explain how the model generalizes to 
$m$-cluster categories of the same type. As an application,
we relate colored diagonals in a polygon 
to semi-standard Young tableaux, in type $E_6,E_7,E_8$.
This provides a new compatibility description of semi--standard
Young tableaux in Grassmannian cluster algebras in type
 $E_6, E_8$ and in a sub-cluster algebra of type $E_7$.
\end{abstract}

\maketitle
\tableofcontents

\addtocontents{toc}{\setcounter{tocdepth}{1}}%

\section{Introduction}

Understanding combinatorial patterns governing 
colored almost positive roots, i.e.~copies of positive roots  together with 
negative simple roots, in root systems is an interesting problem in the 
theory of cluster algebras. This problem has a wide range of implications, 
for instance, finding such descriptions allows to approximate 
similar patterns in larger root systems.
The patterns we are referring to, are the ones arising in generalized cluster complexes and Coxeter combinatorics, or in $m$-cluster categories (also called higher cluster categories) explored in a series of papers, see for example \cite{FR,Th07,Zhu08,BM,BMd}.
Despite various advancements in the field, to describe the usually infinite collection of maximally compatible sets of almost positive roots, resp.~ of colored almost positive roots in general remains a difficult task, well understood only for classical type root systems. 

In this paper, we make some progress in
extending to the exceptional root systems in type $E$ and $F$, 
combinatorial results obtained for the classical types $ABCD$.
Specifically, in the first part (Section 3 and 4) we explore the link between colored almost positive
roots in type $E$ and colored oriented $m$-diagonals in polygons. 
In this way, we generalize to $m$-cluster categories of type $E$ work of Thomas \cite{Th07} and Baur--Marsh \cite{BM}, 
relating $m$-diagonals in a polytope to $m$-cluster categories of type $A$,
and work of Lamberti \cite{L3} describing cluster 
categories of type $E$ in similar geometric terms. 

In the second part, we revisit the relatively simple geometric model of colored diagonals suggested in \cite{L3} and related it to Young tableaux models. More specifically, Jensen, King, and Su in \cite{JKS} gave a description of cluster variables in types $E_6, E_8$ using some combinatorial objects called profiles. We  describe cluster variables in types $E_6, E_7, E_8$ using semi-standard Young tableaux, see Section \ref{sec:tableaux and colored oriented diagonals}. In particular, we present how the geometric model can be used to deduce compatibility of tableaux (Section 5 together with Section 6). Compatibility for almost positive roots is called cluster adjacency in physics, \cite{DFG, LPSV, MSSV} and has important applications to the theory of scattering amplitude in physics. 

\subsection*{Acknowledgements}

B.~Duan is supported by the National Natural Science Foundation of China (no. 11771191). L.~Lamberti would like to thank B.~Zhu for a question asked at the Korean Institute for Advanced Studies during a Conference on Cluster Algebras and Representation Theory in 2014 leading to the geometric model for $m$-cluster categories presented in Section \ref{se:geo2}. J.-R.~Li is supported by the Austrian Science Fund (FWF): M 2633-N32 Meitner Program.

\section{Preliminaries on $m$-cluster categories and Grassmannian cluster algebras}

In this section we define orbit categories of 
representations of valued quivers following \cite{Zhu08} closely.
In this work, however, we will mostly focus on 
representations of ordinary simply-laced Dynkin quivers in types $E_6, E_7$ and $E_8$. 
Results for $m$-cluster categories in type $F_4$ are 
deduced by a folding argument.

\subsection{Definition of $m$-cluster categories}\label{se:rep}
Let $(G,d)$ be a valued graph. That is, 
a finite set of vertices $[n]=\{1, \ldots, n\}$ together with nonnegative integers $d_{ij}$ for all pairs of vertices
$i,j\in [n]\times[n]$ such that $d_{ii}=0$ and positive integers $\{\epsilon_i\}_{i\in G}$ satisfying
\[
d_{ij}\epsilon_j=d_{ji}\epsilon_i 
\]
for all $i,j\in [n]$.
An edge in $(G,d)$ is a pair $\{i,j\}$ with $d_{ij}\neq0$. An orientation $\Sigma$ of $(G,d)$ is given by 
assigning to each edge in $(G,d)$ an order. Denote an oriented valued quiver by
$(G,\Sigma)$. Throughout it is assumed that $(G,\Sigma)$ has no oriented cycles. 
Let $\Phi$ be the root system of the Kac--Moody Lie algebra corresponding to $G$. 
If for each arrow in $(G,\Sigma)$ one has that $\epsilon_j=\epsilon_i$,
$(G,\Sigma)$ is an ordinary quiver, simply denoted by $Q$.

Let $k$ be an algebraically closed field and let $M=(F_i,{}_iM_j)_{i,j\in G}$ be a reduced $k$-species of 
$(G,\Sigma)$. That is, for all $i,j\in [n]$,
${}_iM_j$ is an $F_i-F_j$-bimodule, where $F_i$ and $F_j$
are division rings which are finite dimensional vector spaces over $k$ and $\mathrm{dim}({}_iM_j)=d_{i,j}$ and $\dim_k F_i=\epsilon_i$. Let the category of finite-dimensional
representations of $(G,\Sigma, M)$ be denoted by
$\mathcal{H}$. Let $\mathcal{D}:=\mathcal{D}(\mathcal{H})$ 
be the bounded derived category of
the abelian category $\mathcal{H}$ 
endowed with shift functor
$\Sigma:\mathcal{D}\rightarrow\mathcal{D}$ and 
Auslander--Reiten translation $\tau:\mathcal{D}\rightarrow \mathcal{D}$,
defined as $\mathrm{Hom}_{\mathcal{D}}(X,-)^*\cong\mathrm{Hom}_{\mathcal{D}}(-,\Sigma\tau X),$ for all $X\in\mathcal{D}$. 

For $m\in\mathbb Z_{\geq1}$, define the $m$-cluster category of type $\Phi$ as the orbit category:
\[
\mathcal{C}^m_{\Phi}=\mathcal{D}/\langle \tau^{-1}\Sigma^{m} \rangle.
\]
Objects are the $(\tau^{-1}\Sigma^{m})$-orbits of objects in $\mathcal{D}$ and morphisms are defined by
\[
\mathrm{Hom}_{\mathcal{C}^m_{\Phi}}(\widetilde{X},\widetilde{Y})=\bigoplus_{i\in \mathbb Z}
\mathrm{Hom}_{\mathcal{D}}\Big((\tau^{-1}\Sigma^{m})^iX,Y\Big)
\]
where $\widetilde{X}$ and $\widetilde{Y}$
are representatives of the  $\tau^{-1}\Sigma^{m}$-orbits of $X$ and $Y$ respectively.
In the above, $\Sigma^{m}$ denotes the composition of $\Sigma$ with itself $m$-times.
It is known that $ \mathcal{C}^m_{\Phi}$ is again a triangulated category, \cite[Theorem 1]{Kell08}, with shift $\Sigma$ and Serre functor $\tau\Sigma$  induced from $\mathcal{D}$. Moreover, $\mathcal{C}^m_{\Phi}$ is Krull-Schmidt with finite dimensional $\mathrm{Hom}$-spaces and $(m+1)$-Calabi Yau, see for instance \cite[Prop 2.2]{Zhu08}.

\subsection{Isomorphisms of stable translation quivers}

A quiver $\Gamma$ without loops nor multiple edges, together with a bijective map $\tau:\Gamma\rightarrow\Gamma$, is a
stable translation quiver $(\Gamma,\tau)$, in the sense of \cite{Riedtmann},
if $\tau$ is such that for all vertices $x$ in $\Gamma$
the set of starting points of arrows which end
in $x$ is equal to the set of end points of arrows which start at $\tau(x)$. The map $\tau$
is called translation. For a stable translation quiver
 $(\Gamma,\tau)$ one defines the mesh category of $(\Gamma,\tau)$
as the quotient category of the additive path category of $\Gamma$ by the mesh ideal,
see \cite{K2} for details on this construction.

Let $Q$ be an ordinary quiver. Let $(\mathbb Z Q,\tau)$ be the stable translation quiver given by the repetitive quiver of $Q$, defined as in \cite[I, 5.6]{Happel2}.
Let $\tau: \mathbb Z Q\rightarrow\mathbb Z Q$ be 
defined on the vertices $(n,i)$ of $\mathbb Z Q$ by
$\tau(n,i)=(n-1,i)$, for $n\in\mathbb Z$, $i$ a vertex in $Q$.

We recall, the Auslander--Reiten quiver of any Krull--Schmidt category $\mathcal{K}$ is a
quiver whose vertices are the (isomorphism classes of) indecomposable objects in $\mathcal{K}$.
The number of arrows between two vertices $X$ and $Y$ is given by the
dimension of the space of irreducible morphisms between $X$ and $Y$:
\[
\mathrm{Irr}_{\mathcal K}(X,Y ) := \mathrm{rad}(X, Y)/\mathrm{rad}_2(X,Y).
\]
Here $\mathrm{rad}(X,Y)\subset\mathrm{Hom}_{\mathcal K}(X,Y)$ consists of all non-isomorphisms,
and $\mathrm{rad}_2(X,Y)\subset\mathrm{rad}(X,Y)$ consists
of non-isomorphisms admitting a non-trivial factorization.
Denote the Auslander--Reiten quiver 
of the orbit category $\mathcal{C}^m_{\Phi}$ with 
Auslander--Reiten translation $\tau$ by $(\Gamma_{\Phi}^m,\tau)$.
If $m=1$, the index $m$ will be omitted.

For $Q$ of type $E_6$ let $\rho:\mathbb Z {E_6}\rightarrow\mathbb Z {E_6}$ 
be the order two involution given by the horizontal
reflexion along the center line of $\mathbb Z E_6$.
For $m\in\mathbb Z_{\geq 1},$ let $\rho^m$ be the 
composition of $\rho$ with itself $m$-times.

\begin{lemma}\label{lemma:AR}
Let $Q$ be of type $E_6,E_7,E_8$, then the following are
isomorphisms of stable translation quivers:
\begin{itemize}
\setlength\itemsep{1em}
\item $(\Gamma_{E_6}^{m},\tau)\cong (\mathbb Z E_6/\tau^{-(6m+1)}\rho^{m},\tau)$
\item $(\Gamma_{E_7}^{m},\tau)\cong (\mathbb Z E_7/\tau^{-(9m+1)},\tau)$
\item $(\Gamma_{E_8}^{m},\tau)\cong (\mathbb Z E_8/\tau^{-(15m+1)},\tau)$
\end{itemize}
\end{lemma}
\begin{proof} The claim follows from Happel's result, 
see \cite[I.5.5]{Happel2} together with the description of the
induced action of the functors $\tau$ and $\Sigma$
on $\mathbb Z Q$, first given in \cite[Chap.\ 4]{Miyachi}.
Specifically, the induced action of $\tau$ on $\mathbb Z Q$ is always an
horizontal shift to the left.
The induced action of $\Sigma$ on $\mathbb Z A_n$ can be described by
a shift of $\frac{n+1}{2}$ steps to the right, composed with a reflection
along the horizontal central line of $\mathbb Z A_n$. 
When $n$ is even, the steps are not required to be integer units.
The action of $\Sigma$ on $\mathbb{Z} E_6$
coincides with the action of $\rho\tau^{-6}$ on $\mathbb{Z} E_6$.
Moreover, $\Sigma$ acts as $\tau^{-9}$ on $\mathbb{Z} E_7$,
and as $\tau^{-15}$ on $\mathbb{Z} E_8$.
\end{proof}
The case of $(\Gamma_{F_4}^{m},\tau)$ follows from $(\Gamma_{E_6}^{m},\tau)$ via 
a standard folding argument.

\subsection{Cluster tilting theory in $m$-cluster categories}
Work of Thomas \cite{Th07} (for simply laced cases) and of
Zhu's \cite{Zhu08} (general case) describe the combinatorics of cluster tilting objects in $m$-cluster categories associated to valued quivers. We now recollect a few results from \cite{Th07,Zhu08}.

\begin{de} \leavevmode
\begin{itemize}
\item
An object $T\in\mathcal{C}^m_{\Phi}$ is called $m$-rigid
if it is the direct sum of non isomorphic indecomposable objects
$T_1,\dots, T_t$ such that $\mathrm{Ext}^i_{
\mathcal{C}^m_{\Phi}}(T_j,T_k)=0$ for all
$1\leq i \leq m,$ and $1\leq j,k\leq t$. It is called maximal $m$-rigid if it is $m$-rigid and maximal with respect to this property.
\item
An object $T\in\mathcal{C}^m_{\Phi}$ is an $m$-cluster tilting object which is maximal $m$-rigid and $X\in\mathrm{add} T$ if and only if $\mathrm{Ext}^i_{\mathcal{C}^m_{\Phi}}(T,X)=0$
for $1\leq i \leq m$.
\end{itemize}
\end{de}

To exhaustively search for $m$-cluster tilting objects in $m$-cluster categories, it is crucial to know how many objects of this type there are in a fixed $m$-cluster category. The next result gives an explicit answer for $m$-cluster categories associated to valued quivers.
 
\begin{thm}[\cite{FR,Zhu08}]\label{thm:Zhu_FR}
Let $Q$ be an orientation of a connected 
Dynkin diagram and $\Phi$ 
the corresponding root system. The number of $m$-cluster 
tilting objects in $\mathcal{C}^m_{\Phi}$ is
\[
\prod_i\frac{mh+e_i+1}{e_i+1}
\]
where $h$ is the Coxeter number of $\Phi$, and $e_1,\dots, e_n$ 
are the exponents of $\Phi$.
\end{thm}
Theorem \ref{thm:Zhu_FR}
follows from \cite[Prop. 8.4]{FR} as described in \cite[Cor.~5.8]{Zhu08}. The exponents in the next Table are extracted from \cite{FR}.
\begin{center}
\begin{tabular}{ |c|c|c| c|}
 \hline
 $\Phi$  & $h$ & Exponents & $2$-cluster tilting objects\\
  \hline
 $E_6$  & 12 & 1,4,5,7,8,11 & 16.588 \\
 $F_4$  & 12 & 1,5,7,11&  780\\
 $E_7 $ & 18 & 1,5,7,9,11,13,17 & 144.210\\
 $E_8$  & 30 & 1,7,11,13,17,19,23,29 &15.209.220 \\
 \hline
\end{tabular}
\end{center}

\subsection{Colored almost positive roots}\label{sec:roots}
Let $\Phi$ be an irreducible root system of rank $n$, associated to a valued graph, or quiver.
Let $I$ be an $n$-element indexing set.
Let $\Phi_{>0}$ be the set of positive roots and let $\{\alpha_i:i\in I\}$ be the set of simple roots in $\Phi$.
Following \cite{FR} we define the set of almost positive roots in $\Phi$ as
\[
\Phi_{\geq -1} = \Phi_{>0} \ \cup \ \{-\alpha_i:i\in I\}.
\]
For $m\in\mathbb Z_{>0}$, the set of colored almost positive roots consists of $m$-copies of the set $\Phi_{>0}$ together with one copy of the negative simple roots in $\Phi$, i.e.:
\[
\Phi^m_{\geq -1} = \underbrace{\Phi_{>0} \ \cup \ \dots \ \cup \ \Phi_{>0} }_{m-\textrm{copies}} \ \cup \ \{-\alpha_i:i\in I\}.
\] 
When $m=1$, the index in $\Phi^m_{\geq -1}$ is omitted. The compatibility
degree of $m$-colored almost positive roots is defined in \cite[Def 3.1]{FR}.

On the one side, work of Thomas \cite[\S.6]{Th07} describes a correspondence between $m$-colored almost positive roots and indecomposable objects in a geometrically defined $m$-cluster category associated to $\Phi$. The compatibility
degree of $m$-colored almost positive roots can then equivalently be computed via $\mathrm{Ext}$-functors in 
$m$-cluster categories \cite[Prop.~2]{Th07}. 
The non-simply laced case was described in work of Zhu \cite[\S.5]{Zhu08} and \cite[\S.5, Def.~5.2]{Zhu08}.
When $m=1$ and $Q$ is of finite Dynkin type, 
this property is shown in
\cite[Prop.~4.2]{BMRRT}, \cite{Zhu06}.

On the other side, 
work of Fomin--Reading \cite[\S.5]{FR} implies that 
there is a one-to-one correspondence 
between $m$-diagonals in certain polygons $\Pi$ and $m$-colored almost positive roots $\Phi_{\geq -1}^m$ of type $ABCD$.
See also work of Tzanaki  \cite{Tz06} for types $AB.$ 

In the above settings, two colored almost positive roots  $\alpha$ and $\beta$ are compatible, if the corresponding $m$-diagonals $D_{\alpha}$ and $D_{\beta}$ do not cross.
When $m=1,$ these descriptions coincide with the ones given in \cite[\S.3.5]{FZy}.

\subsection{The $m$-power of a translation quiver}
Colored almost positive roots in a root system of type $A$ can be identified with a subset of almost positive roots in a larger root system of the same type,
see for instance \cite{Th07,FR,Tz06}.
In \cite{BM} this relationship
is described in terms of an $m$-power operation. 
We now extend this approach 
to $m$-cluster categories of type $E$.

\begin{de}
Let $(\Gamma, \tau)$ be a translation quiver. The quiver $(\Gamma)^m$ 
whose vertices are the same as the ones from $\Gamma$
and whose arrows are sectional paths of length $m$ is called 
the $m$-th power of $\Gamma$. A path 
$(x=x_0\rightarrow x_1\rightarrow...\rightarrow x_{m-1}\rightarrow x_m=y)$ in $\Gamma$ is said to be
sectional if $\tau x_{i+1}\neq x_{i-1}$ for $i=1,...,m-1$ (for which $\tau x_{i+1}$ is defined).
\end{de}

\begin{lemma}\label{lemma_thm6.1}
If $(\Gamma, \tau)$ is a stable translation quiver, then
$((\Gamma)^m,\tau^m)$ for $\tau^m:=\tau\circ...\circ\tau,$
composed $m$ times, is again a stable translation quiver.
\end{lemma}
\begin{proof}
This is shown in Theorem 6.1 in \cite{BM} and Corollary 6.2 in \cite{BM}.
\end{proof}

Let $T_{r,s,t}$ be an orientation of a finite connected
graph with $r+s+t+1$ vertices and three legs.
Assume that one vertex of $T_{r,s,t}$ has three neighbors and that
the legs of $T_{r,s,t}$ have $r$, resp.~$s$, 
resp.~$t$ vertices. A tree $T_{r,s,t}$
is symmetric if two legs have the same number of vertices.

Below, for $i\in\{6,7,8\}$, let $(\Gamma_{E_i}^{m},\tau)$ be as described in Lemma \ref{lemma:AR}.

\begin{prop}\label{prop_m_power}
The following quivers are connected components of the
$m$-power of the given stable translation quivers:
\begin{itemize}
 \setlength\itemsep{1em}
\item $
(\Gamma_{E_6}^{m},\tau^m)\subseteq(\mathbb Z T_{2m-1,2m,2m}/\tau^{-(6m+1)}\rho^m)^m$
\item $(\Gamma_{E_7}^{m},\tau^m)\subseteq(\mathbb Z T_{2m-1,2m,3m}/\tau^{-(9m+1)})^m$
\item $(\Gamma_{E_8}^{m},\tau^m)\subseteq(\mathbb Z T_{2m-1,2m,4m}/\tau^{-(15m+1)})^m$
\end{itemize}
\end{prop}
\begin{proof}
Lemma \ref{lemma:AR} describes the shapes
of the quiver $(\Gamma_{E_i}^{m},\tau)$, for $i\in\{6,7,8\}$. 
Reversing the $m$-power procedure, by adding meshes in
the quivers $\Gamma_{E_i}^{m}$, for $i\in\{6,7,8\}$ gives
rise to the larger quivers in the claim.
\end{proof}

\subsubsection{Geometric categories
associated to tree diagrams}\label{se:geo1}

Let $T_{r,s,t}$ be a fixed tree diagram.
For $n\geq\mathrm{max}\{r+s+1;r+t+1\}$ let
$\Pi_{r,s,t}$ be a regular $(n+3)$-gon with vertices numbered 
in the clockwise order by the group $\mathbb Z/(n+3)\mathbb Z$.

For the vertices $i,j,k$ in $\Pi_{r,s,t}$ we write $i\leq j\leq k$
if the vertex $j$ is between $i$ and $k$ in the clockwise order.
A diagonal in $\Pi_{r,s,t}$
joining the vertices $i$ and $j$ is denoted by $(i,j)$
and an oriented diagonal
in $\Pi_{r,s,t}$ starting at $i$ and ending 
in $j$ is denoted by
$[i,j]$. Boundary segments are 
not considered.

Following \cite{L3}, we associate to each leg 
in $T_{r,s,t}$ a set
of diagonals in $\Pi_{r,s,t}$. 
To do so, we double the set of all oriented diagonals in $\Pi_{r,s,t}$,
and distinguish them with colors, red and blue, and
subscripts $R$, $B$.
Specifically, for every vertex $i$ in $\Pi_{r,s,t}$ we then form
$(r+1)$ pairs of
colored oriented diagonals:
\begin{align*}\itemsep0em
[i,i+2]_P=&[i+2,i]_P:=\{ [i,i+2]_R,[i+2,i]_B\}\\
[i,i+3]_P=&[i+3,i]_P:=\{ [i,i+3]_R,[i+3,i]_B\}\\
&\dots\\
[i,i+r+2]_P=&[i+r+2,i]_P:=\{ [i,i+r+2]_R,[i+r+2,i]_B\}.
\end{align*}
As before, coordinates are viewed as elements of 
$\mathbb Z/(n+3)\mathbb Z.$
Once colored oriented
diagonals are paired, they stop existing as single
diagonals in $\Pi_{r,s,t}$.
In the following, we color paired diagonals
in green.

Next, one defines a subset of the diagonals 
of $\Pi_{r,s,t}$ consisting of
$(r+1)(n+3)$ paired, $s(n+3)$ single 
red and $t(n+3)$ single blue oriented
diagonals:
\begin{align*}\itemsep0em
P_{r,s,t}:=\bigg\{&[i,i+2]_G, \dots,[i,i+r+2]_G, \\
&[i,i+r+3]_R,\dots,[i,i+r+s+2]_R, \\
&[i+r+3,i]_B, \dots,[i+r+t+2,i]_B,\hspace{0,3cm} \textrm{ $i$ vertex  of }\Pi_{r,s,t}\bigg\}.
\end{align*}
The set of red, blue and green diagonals of $\Pi_{r,s,t}$
are disjoint. Consequently, orientations can now be omitted. 
Elements in $P_{r,s,t}$ are henceforth called colored diagonals.

\subsubsection{Automorphisms on 
the set of colored diagonals}\label{rhotau}

Let $\rho, \tau_0$ and $\tau$ be automorphisms
acting on the set of colored diagonals $P_{r,s,t}$
of  $\Pi_{r,s,t}$ associated to a tree
$T_{r,s,t}$. The first two automorphisms are non-trivial only when $T_{r,s,t}$ is symmetric. While the 
automorphism $\tau:P_{r,s,t}\rightarrow P_{r,s,t}$ is always
given as the anti-clockwise rotation through 
$\frac{2\pi}{n+3}$ around the center of $\Pi$.

For $c\in\{R,B,G\}$, let $\rho:P_{r,t,t}\rightarrow P_{r,t,t}$
be the automorphism of order two
given by
$$
\rho\big([i,j]_c\big):=\begin{cases}
[j,i]_B & \textrm{if } c=R, \\
[j,i]_R  &\textrm{if } c=B,\\
[i,j]_G &\textrm{otherwise.}
\end{cases}
$$
The definition of $\tau_0:P_{r,t,t}\rightarrow P_{r,t,t}$ is designed to accommodate the
action of the shift functor on the $m$-cluster category 
of type $E_6$:
\[
\tau_0([i,j]_c):=\begin{cases}
\rho\big([i-1,j-1]_c\big)  &\textrm{ if }[i,j]_c\in  P_{r,t,t}\vert_1 \\
[i-1,j-1]_c & \textrm{ otherwise.}
\end{cases}
\]
Here $P_{r,t,t}\vert_1\subset P_{r,t,t}$ denotes the set of diagonals starting, resp.~ending, at the fixed vertex $i=1$.

\subsubsection{Minimal clockwise rotations}\label{rotations}
The definition of minimal clockwise rotation of elements
of $\Pi_{r,s,t}$ changes in the presence of $\rho$.
First, consider the case where $T_{r,s,t}$ is not symmetric, 
or when $T_{r,t,t}$ is symmetric and $\tau$
is the standard anti-clockwise rotation. Then the
minimal clockwise rotations are the standard ones.
That is, for $k,j$ non-neighboring vertices and for $c\in\{R, B, G\}$, 
the minimal clockwise rotation among diagonals in
$P_{r,s,t}$ is defined by:
$[k,j]_c\rightarrow [k,j+1]_c$ and $[k,j]_c\rightarrow[k+1,j]_c$
and if $j\in\{k+r+2,k+r+3\}$ by:
\[
\xymatrix@R=0.5pc @C=0.5pc{
&&[k,k+r+3]_R\ar[rd]&&\\
&[k,k+r+2]_G \ar[ru]\ar[rd]& &[k+1,k+r+3]_G. \\
&&[k+r+3,k]_B\ar[ur]&&
} .\] 

Second, when $T_{r,t,t}$ is symmetric and the translation is
$\tau_0$, 
the minimal clockwise rotations are defined as before
except when $[k,j]_c $,
$[k,k+r+2]_G$ are in $\tau_0(P_{r,t,t}\vert_1)$, 
then one combines the minimal clockwise rotation with
$\rho$.

\subsubsection{Quivers of single and paired colored oriented
diagonals}\label{stabtrans}
Let $\Gamma^{n+3}_{r,s,t}$ be the quiver whose vertices 
are the elements of $P_{r,s,t}$. 
An arrow between two vertices of $\Gamma^{n+3}_{r,s,t}$
is drawn whenever there is a minimal clockwise rotation linking them. No arrow is drawn otherwise.
In this way $\Gamma^{n+3}_{r,s,t}$ lies on a cylinder, unless
when $s=t$ and the translation is given by $\tau_0$. Then 
$\Gamma^{n+3}_{r,t,t}$ lies on a M\"obius strip.

\begin{lemma}\cite[Lemma 4.1]{L3} \label{lemma:stable}
The quivers $(\Gamma_{r,s,t}^{n+3},\tau)$ and
$(\Gamma_{r,t,t}^{n+3},\tau_0)$ are stable translation quivers. \qed
\end{lemma}

\subsubsection{Folded quiver $\Gamma^{n+3}_{r,t}$}
When the quiver $({\Gamma}^{n+3}_{r,t,t},\tau)$, 
resp.~ $({\Gamma}^{n+3}_{r,t,t},\tau_0)$, 
are symmetric along its horizontal central line, 
folding produces a new translation quiver which we denote by
$(\Gamma^{n+3}_{r,t},\tau)$. More precisely, consider again the graph automorphism $\rho$.
The vertices of $\Gamma^{n+3}_{r,t}$ are the
$\rho$-orbits of vertices of $\Gamma^{n+3}_{r,t,t}$.
The arrows in $\Gamma^{n+3}_{r,t}$ are always single and
coincide with minimal clockwise rotation around a common vertex of $\Pi_{r,t,t}$
linking pairs of colored diagonals.
The translation $\tau$ on $\Gamma^{n+3}_{r,t}$ is the induced 
one and given by the anti clockwise rotation through 
$\frac{2\pi}{n+3}$ around the center of $\Pi_{r,t,t}$.

\subsection{Semi-standard Young tableaux and the BFZMS twist}

Denote the set of semi-standard Young tableaux of rectangular shape with $n$ rows and with entries in $[m] = \{1,\dots,m\}$ by ${\rm SSYT}(n, [m])$. We will write a tableau as a matrix.

In \cite{CDFL}, it is shown that every cluster variable in $\mathbb{C}[{\rm Gr}(n,m)]$ corresponds to some tableau $T \in {\rm SSYT}(n, [m])$. Note that not every tableau $T \in {\rm SSYT}(n, [m])$ corresponds to some cluster variable. The tableaux corresponding to cluster variables form a proper subset of ${\rm SSYT}(n, [m])$. We denote the cluster variable corresponding to $T$ by ${\rm ch}(T)$. 
An explicit formula for ${\rm ch}(T)$ is given in \cite{CDFL}. We recall the formula ${\rm ch}(T)$.

In this paper, we are interested in computing the BFZMS twist (removing all frozen factors) of a cluster variable.
It is proved in \cite[Theorem 6]{GLS}, \cite{MS} that the Auslander--Reiten translation $\tau$ corresponds to a twist map (we call it BFZMS twist and denote it also by $\tau$) on $\mathbb{C}[{\rm Gr}(n,m)]$ defined by Berenstein, Fomin, and Zelevinsky in \cite{BFZ96, BZ} and Marsh and Scott in \cite{MS}. Therefore it suffices to use the following version of the formula for ${\rm ch}(T)$. 

Denote by $\mathbb{C}[{\rm Gr}(n,m,\sim)]$ the quotient of $\mathbb{C}[{\rm Gr}(n,m)]$ by the inhomogeneous ideal 
\begin{align}
\langle P_{i,i+1, \ldots, i+n-1}-1, \quad i \in [m-n+1] \rangle.
\end{align}

For $S,T \in {\rm SSYT}(n, [m])$, we denote by $S \cup T$ the row-increasing tableau whose $i$th row is the union of the $i$th rows of $S$ and $T$ (as multi-sets). It is shown in \cite{CDFL} that $S \cup T$ is semi-standard if $S, T$ are semi-standard.

We call $S$ a factor of $T$ and write $S \subset T$ if the $i$th row of $S$ is contained in that of $T$ (as multi-sets), for every $i \in \{1, \ldots, n\}$. A tableau $T \in {\rm SSYT}(n, [m])$ is called trivial if each entry of $T$ is one less than the entry below it. 

An equivalence relation $\sim$ on ${\rm SSYT}(n, [m])$ is defined in \cite{CDFL} as follows. For any $T \in {\rm SSYT}(n, [m])$, we  denote by $T_{\text{red}}$ the tableau with the minimal number of columns such that $T = T_{\text{red}} \cup S$ for a trivial tableau $S$. For $S, T \in {\rm SSYT}(n, [m])$, define $S \sim T$ if $S_{\text{red}} = T_{\text{red}}$. We use the same notation for a tableau $T$ and its equivalence class.

For a Pl\"{u}cker coordinate $P=P_{i_1, \ldots, i_n} \in \CC[\Gr(n,m)]$, the gap weight of $P$ is defined as $\sum_{j=2}^n(i_j-i_{j-1}-1)$, \cite{CDFL}. The gap weight of a tableau $T$ is the sum of the gap weights of $P_{T_i}$, $i \in \{1,\ldots,k\}$, where $T_i$'s are columns of $T$. A tableau $T \in {\rm SSYT}(n,[m])$ has small gap if each of its columns has gap weight exactly $1$. For every $T \in {\rm SSYT}(n,[m])$, there is a unique small gap tableau $T' \in {\rm SSYT}(n,[m])$ such that $T \sim T'$. 

Given $T \in {\rm SSYT}(n,[m])$ with gap weight $k$, let $T' \in {\rm SSYT}(n,[m])$ be the small gap tableau equivalent to $T$. Let ${\bf i}_T = i_1 \leq i_2 \dots \leq i_k$ be the entries in the first row of $T'$. Then the $a$th column of $T'$ has content $[i_a,i_a+n] \setminus \{r_a\}$ for some $r_a$. Let ${\bf j}_T = j_1 \leq j_2 \leq \dots \leq j_k$ be the elements $r_1,\dots,r_k$ written in weakly increasing order. There is a unique maximal length $w_T \in S_k$ such that $\{[i_{w_T(a)}, i_{w_T(a)}+n] \setminus \{j_a\}\}_{a \in \{1, \ldots, k\}}$ are the entries of columns of $T'$. 

Let $T$ be a small gaps tableau with $k$ columns, with ${\bf i} = {\bf i}_T, {\bf j}={\bf j}_T \in \ZZ^k$ the weakly increasing sequences just defined. For $u \in S_k$, define $P_{u;T} \in \CC[\Gr(n,m)]$ as follows. If $j_a \in [i_{u(a)}, i_{u(a)}+n]$ for all $a \in [k]$, then the tableau $\alpha(u;T)$ is the semi-standard tableau whose columns have content $[i_{u(a)}, i_{u(a)}+n] \setminus \{j_a\}$ for $a \in [k]$, and $P_{u; T} = P_{\alpha(u;T)} \in \CC[\Gr(n,m)]$ to be the corresponding standard monomial. If $j_a \notin [i_{u_a}, i_{u(a)}+n]$ for some $a$, then the tableau $\alpha(u;T)$ is undefined and $P_{u ;T} = 0$. 

Let $T \in {\rm SSYT}(n,[m])$ with gap weight $k$ and let $T' \sim T$ be the small gap tableau equivalent to $T$. Then 
\begin{align}
\ch(T) = \sum_{u \in S_k} (-1)^{\ell(uw_T)} p_{uw_0, w_Tw_0}(1) P_{u; T'}, 
\end{align}
where $p_{u,v}(q)$ is the Kazhdan--Lusztig polynomial \cite{KL}. 

There is an order ``$\ge$'' called dominance order on the set of partitions. For two partitions $\lambda = (\lambda_1,\dots,\lambda_\ell)$ and $\mu = (\mu_1,\dots,\mu_\ell)$ with $\lambda_1 \geq \lambda_2 \ge \cdots \geq \lambda_\ell \geq 0$ and $\mu_1 \geq \mu_2 \ge \cdots \geq \mu_\ell \geq 0$, $\lambda \geq \mu$ if and only if $\sum_{j \leq i}\lambda_j \geq \sum_{j \leq i}\mu_j$ for $i=1,\dots,\ell$. 

For a tableau $T$, let ${\rm sh}(T)$ denote the shape of $T$. For $i \in [m]$, denote by $T[i]$ the restriction of $T \in {\rm SSYT}(n,[m])$ to the entries in $\{1, \ldots, i\}$. There is a natural order on the set ${\rm SSYT}(n,[m])$: for $T,T' \in {\rm SSYT}(n, [m])$,
\begin{align}
T \geq T' \text{ if and only if for $i=1,\dots,m$, } {\rm sh}(T[i]) \geq {\rm sh}(T'[i]) \label{eq: order for tableaux}
\end{align}
in the dominance order on partitions.
 
For a one-column tableau $T$, denote by $P_T$ the Pl\"{u}cker coordinate with indices which are the entries of $T$. For $T \in {\rm SSYT}(n, [m])$ with columns $T_1,\dots,T_k$, let $P_T \in \mathbb{C}[{\rm Gr}(n,m)]$ denote the monomial in Pl\"{u}cker coordinates $P_{T_1} \cdots P_{T_k}$.

The monomials $\{P_T\}$, where $T \in {\rm SSYT}(n,[m])$, are a basis for $\mathbb{C}[{\rm Gr}(n, m)]$ known as the standard monomial basis \cite{Ses}. Thus any $x \in \mathbb{C}[{\rm Gr}(n, m)]$ can be written as $x = \sum_{T} c_T P_T$ for some $T \in {\rm SSYT}(n,[m])$ and $c_T \in \CC^{\times}$. It is shown in \cite{CDFL} that the map ${\rm Top}(x)$ taking the largest tableaux in $x = \sum_{T} c_T P_T$: 
\[
{\rm Top}(x) = \max\{T: T \text{ appears in } x=\sum_{T} c_T P_T\}
\]
is well defined.

\begin{example}
\begin{align*}
{\rm Top}( P_{124}P_{356} - P_{123}P_{456} ) = \left(\begin{array}{cc} 1&3\\ 2&5\\ 4&6 \end{array}\right).
\end{align*}
\end{example}

For $v_1, \ldots, v_{k-1} \in \mathbb{C}^k$, the generalized cross product $v_1 \times \cdots \times v_{k-1}$ is the unique vector in $\mathbb{C}^k$ such that 
\begin{align*}
(v_1 \times \cdots \times v_{k-1})^T v = \det(v_1, \ldots, v_{k-1}, v).
\end{align*}

Marsh and Scott defined a twist map $\tau$ on $\mathbb{C}[{\rm Gr}(n, m)]$, \cite{MS}. For a $n \times m$ matrix $p$, the twist $p'$ of $p$ is the $n \times m$ matrix whose $i$th column vector is 
\begin{align*}
(p')_i = \epsilon_i \cdot (p_{\sigma^{k-1}(i)} \times p_{\sigma^{k-2}(i)} \times \cdots \times p_{\sigma(i)}),
\end{align*}
where $\sigma: \{1, \ldots, n\} \to \{1, \ldots, n\}$ is the map given by $i \mapsto i-1 \pmod n$ and
\[
\epsilon_i = \begin{cases} (-1)^{i(k-i)}, & i \le k-1, \\ 1, & i \ge k. \end{cases}
\]

The twist map on the set of $n \times m$ matrices induces a twist map $\tau$ on $\mathbb{C}[{\rm Gr}(n, m)]$. It is shown in \cite[Proposition 8.10]{MS} that $\tau$ sends a cluster variable to a cluster variable (possibly multiply by some frozen variables). 

All mutations in a Grassmannian cluster algebra can be described using tableaux, \cite{CDFL}. Starting from the initial seed of $\CC[\Gr(n,n+\ell+1,\sim)]$, at each mutation step, when one mutates at the vertex $k$ with a cluster variable $\ch(T_k)$, one obtains a new cluster variable $\ch(T'_k)$, where 
\begin{align*}
\ch(T_k)\ch(T'_k) = \prod_{k \to i} \ch(T_i) + \prod_{i \to k} \ch(T_i), \quad T'_k = T_k^{-1} \max\{\cup_{i \to k} T_i, \cup_{k \to i} T_i \},
\end{align*}
and $\max$ is to take the largest tableau with respect to ``$\le$'' defined in (\ref{eq: order for tableaux}) and $A^{-1}B$ is the tableau obtained by deleting the elements in the $j$th row of $A$ from that of $B$ (as multi-sets), $j \in [n]$.

\section{Geometric $m$-cluster categories of type $F_4$, $E_6$, $E_7$, $E_8$} \label{se:geo2}

The geometric model we propose in this section relates
to the geometric model for $m$-cluster categories of type $A$ 
given by Thomas in \cite{Th07} and independently by Baur--Marsh in \cite{BM}.

We then describe how colored almost positive roots
associated to the $m$-cluster categories of type $E$
relate to both almost positive roots of 
cluster categories associated to larger tree diagrams and  
to almost positive roots associated to repetitive cluster categories of type $A$.
In this way, we find new links among the geometric models describing the above categories.

\subsection{The quiver of $m$-diagonals}

Consider the set $P_{1,k_1,k_2}^{n+3,m}\subseteq P_{r,s,t}$ of $m$-colored diagonals given by
\begin{align*}\itemsep0em
P_{1,k_1,k_2}^{n+3,m}:=\bigg\{&[i,i+r+2-m]_G,[i,i+r+2]_G, \\
&[i,i+r+2+m]_R,[i,i+r+2+2m]_R,\dots,[i,i+r+2+k_1m]_R, \\
&[i+r+2+m,i]_B, [i+r+2+2m,i]_B\dots,[i+r+2+k_2m,i]_B:\\
& \textrm{ $i$ vertex  of }\Pi_{r,s,t}\bigg\} \enspace.
\end{align*}
As before, coordinates are viewed 
as elements of $\mathbb Z/(n+3)\mathbb Z$.
For $m\in\mathbb Z_{>0}$, let $\Gamma^{n+3,m}_{1,k_1,k_2}$ be the quiver with vertices given by the
elements of $P_{1,k_1,k_2}^{n+3,m}$
and with arrows induced from ${\Gamma}^{n+3}_{r,s,t}$.
Let $\tau^m: \Gamma^{n+3,m}_{1,k_1,k_2}\rightarrow\Gamma^{n+3,m}_{1,k_1,k_2}$,
resp.~$\tau_0^m: \Gamma^{n+3,m}_{1,k_1,k_1}\rightarrow\Gamma^{n+3,m}_{1,k_1,k_1}$
be the translation induced from $({\Gamma}^{n+3}_{r,s,t},\tau)$, resp.~from $(\Gamma^{n+3,m}_{r,t,t},\tau_0).$

\begin{lemma}\label{lemma:isoquivers}
The following are isomorphisms of stable translation quivers:
\begin{itemize}
\setlength\itemsep{1em}
\item $(\Gamma_{1,2,2}^{6m+1,m},\tau_0^m)\cong (\mathbb Z E_6/\tau^{-(6m+1)}\rho^{m},\tau)$
\item $(\Gamma_{1,2,3}^{9m+1,m},\tau^m)\cong (\mathbb Z E_7/\tau^{-(9m+1)},\tau)$
\item $(\Gamma_{1,2,4}^{15m+1,m},\tau^m)\cong (\mathbb Z E_8/\tau^{-(15m+1)},\tau)$
\end{itemize} \qed
\end{lemma}
Denote
the mesh category associated to the translation quiver $(\Gamma,\tau)$
by $\mathcal{C}(\Gamma,\tau)$.

\begin{thm}\label{thm:geom_equiv}
The following are equivalences of additive categories:
\begin{align*}
&\mathcal{C}(\Gamma_{1,2,2}^{6m+1,m},\tau_0^m)\rightarrow \mathcal{C}^{m}_{E_6} \\
&\mathcal{C}(\Gamma_{1,2,3}^{9m+1,m},\tau^m)\rightarrow \mathcal{C}^{m}_{E_7}\\
&\mathcal{C}(\Gamma_{1,2,4}^{15m+1,m},\tau^m)\rightarrow \mathcal{C}^{m}_{E_8}.
\end{align*}
\end{thm}
\begin{proof}
Given Lemma \ref{lemma:stable} 
all we need to show is that there is an isomorphism of stable translation quivers
between $({\Gamma}^{n+3,m}_{r,s,t}, \tau_l^m)$ and
the Auslander--Reiten quiver of $\mathcal{C}^m_{E_u}$, for
the appropriate values of $l,n,r,s,t$ and $u\in\{6,7,8\}$. 
For the shape of the Auslander--Reiten quiver of $\mathcal{C}^m_{E_u}$,
see Lemma \ref{lemma:AR}.
\end{proof}
\begin{remark}
When $m=1$, Theorem \ref{thm:geom_equiv} becomes Theorem 4.3 in \cite{L3}.
\end{remark}
Two slices from the quivers $\mathcal{C}(\Gamma_{1,2,2}^{7},\tau_0)$, $\mathcal{C}(\Gamma_{1,2,3}^{10},\tau)$ and $\mathcal{C}(\Gamma_{1,2,4}^{16},\tau)$ are illustrated in Figure \ref{fig: AR quiver E6} and Figure \ref{fig: AR quiver E7}. The remaining slices are obtained from these by applying the translation map. 

\begin{figure}
$$
\scalemath{0.8}{
\xymatrix@-6mm@C-0.2cm{
& & [1,6]_R \ar[rdd] & &    [2,7]_R    \\
\\
&  [1,5]_R   \ar[rdd] \ar[ruu] & &  [2,6]_R   \ar[ruu]\ar[rdd] \\
 \\
&  \ar[r]  [2,4]_G    & \ar[r]   [2,5]_G \ar[ruu]\ar[rdd] &   [3,5]_G \ar[r]  & [3,6]_G \\
 \\
& [5,1]_B    \ar[ruu] \ar[rdd]&&   [6,2]_B \ar[rdd] \ar[ruu]  \\
 \\
&&  [6,1]_B \ar[ruu]&&  [7,2]_B \\
} }
$$
\caption{Two slices of the Auslander--Reiten 
quiver of $\mathcal{C}_{E_6}$}
\label{fig: AR quiver E6}
\end{figure}
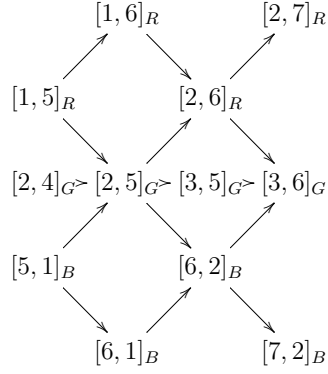

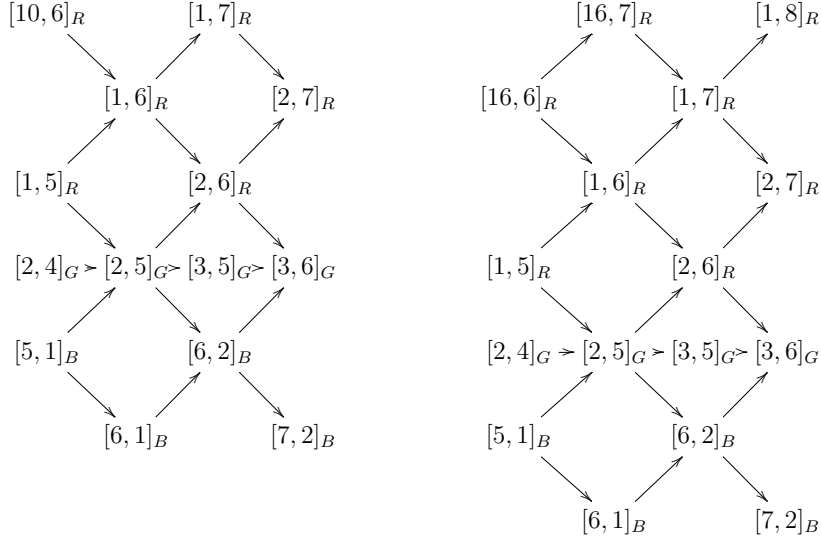
\begin{figure}
$$
\scalemath{0.8}{
\xymatrix@-6mm@C-0.2cm{
& [10,6]_{R} \ar[rdd] & & [1,7]_R \ar[rdd] & &       \\
\\
& & [1,6]_R  \ar[ruu] \ar[rdd] & &    [2,7]_R  & &  
\\
\\
&  [1,5]_R   \ar[rdd] \ar[ruu] & &  [2,6]_R    \ar[rdd]\ar[ruu] & &    
 \\
 \\
&  \ar[r]  [2,4]_G    & \ar[r]   [2,5]_G \ar[ruu]\ar[rdd] &  \ar[r]  [3,5]_G    &[3,6]_G &     &
 \\
 \\
& [5,1]_B    \ar[ruu] \ar[rdd]&&   [6,2]_B  \ar[ruu] \ar[rdd]&&    \\
 \\
&&  [6,1]_B \ar[ruu]&&  [7,2]_B  &&   \\
} }
\hspace{1cm}
\scalemath{0.8}{
\xymatrix@-6mm@C-0.2cm{
& & [16,7]_R \ar[rdd] & &    [1,8]_R       \\
\\
& [16,6]_{R} \ar[ruu] \ar[rdd] & & [1,7]_R \ar[ruu] \ar[rdd] & &      \\
\\
& & [1,6]_R  \ar[ruu] \ar[rdd] & &    [2,7]_R  \\
\\
&  [1,5]_R   \ar[rdd] \ar[ruu] & &  [2,6]_R    \ar[rdd]\ar[ruu] & &        \\
 \\
&  \ar[r]  [2,4]_G    & \ar[r]   [2,5]_G \ar[ruu]\ar[rdd] &  \ar[r]  [3,5]_G    & [3,6]_G &   \\
 \\
& [5,1]_B    \ar[ruu] \ar[rdd]&&   [6,2]_B  \ar[ruu] \ar[rdd] \\
 \\
&&  [6,1]_B \ar[ruu]&&  [7,2]_B     \\
} }
$$
\caption{Two slices of the Auslander--Reiten 
quiver of $\mathcal{C}_{E_7}$ (on the left) and of $\mathcal{C}_{E_8}$ (on the right).}
\label{fig: AR quiver E7}\label{fig: AR quiver E8 part 1}
\end{figure}

\subsection{Colored almost positive roots, and colored $m$-diagonals}\label{se:roots}

In Prop.~\ref{prop:mcatroots} we explain how the non-crossing property of $m$-diagonals in polygons, resp.~of compatibility of colored almost positive roots in type $A$, translates to geometric $m$-cluster categories of type $E$. 
From this result, it will follow that
maximal sets of non-crossing $m$-colored diagonals 
describe only a subset of all possible 
maximal compatible sets of colored 
almost positive roots in type $E$. To describe the remaining
sets is much more difficult. 
In Section \ref{sec:smallex} and Section \ref{sec:ct} further descriptions of these sets will be provided.

\begin{prop}\label{prop:mcatroots}
For any pair of $m$-colored almost positive real Schur roots $\alpha, \beta$ of type $\Phi\in \{E_6,E_7,E_8\}$ the compatibility degree $(\alpha\vert\vert \beta)_{\mathcal{ C}^m_{\Phi}}$ agrees with the number of $m$-curves of $D_{\alpha}$ intersecting  with $D_{\beta}$.
\end{prop}
\begin{proof}
(Sketch). For type $E_6$ and when $m=1$ the claim is shown in \cite[Prop.5.3]{L3}. For $m$-cluster categories every such root corresponds to an indecomposable object in $\mathcal{ C}^m_{\Phi}$. The compatibility degree of $\alpha, \beta$ then agrees with the dimension of 
$\mathrm{Ext}_{\mathcal{ C}^m_{\Phi}}(I_\alpha,I_\beta)$ as described in \cite[Prop. 2]{Th07}. Here
$I_\alpha,I_\beta \in\mathcal{ C}^m_{\Phi}$ are the
indecomposable objects parametrized by $\alpha, \beta$. The dimension of these groups can then be found by computing the support of $\mathrm{Ext}_{\mathcal{ C}^m_{\Phi}}(-,-)$ on the AR-quiver of $\mathcal{C}_{\Phi}^m$. 
This is done via starting, resp.~ending functions, as described in \cite[\S.8.]{BMRRT}.
By Thm.~\ref{thm:geom_equiv} the AR-quiver of 
$\mathcal{ C}_{\Phi}^m$ is isomorphic to the quiver of $m$-diagonals.
This implies that there is a geometric description of the support of the $\mathrm{Ext}$-functor. 
This amounts in specifying rotations for the $m$-diagonal $D_{\alpha}$ 
through curves and generalizes the approach of \cite[\S.5]{L3} to the setting of $m$-diagonals.  
\end{proof}

For a regular $(n+3)$-gon $\Pi,$ $n\geq\mathrm{max}\{r+s+1;r+t+1\},$ let $({\Gamma}^{n+3}_{r,s,t}, \tau)$
be the stable translation quiver of diagonals
associated to colored diagonals in $\Pi$ defined above.

\begin{thm}\label{thm:mpower}
The following quivers are connected components of the
 $m$-power of the given stable translation quivers:
\begin{itemize}
  \setlength\itemsep{1em}
\item $
(\Gamma_{E_6}^{m},\tau_0^m)\subseteq(\Gamma_{2m-1,2m,2m}^{(6m+1)},\tau_0)^m$
\item $(\Gamma_{E_7}^{m},\tau^m)\subseteq(\Gamma_{2m-1,2m,3m}^{(9m+1)},\tau)^m$
\item $(\Gamma_{E_8}^{m},\tau^m)\subseteq(\Gamma_{2m-1,2m,4m}^{(15m+1)},\tau)^m$
\end{itemize}
\end{thm}
\begin{proof}
The claim follows by combining Prop.~\ref{prop_m_power} with Lemma \ref{lemma:isoquivers}. Notice
that the components need not to be unique.
\end{proof}
Two observations are in order. First, Theorem \ref{thm:mpower} implies that colored almost positive roots in type $E_6,E_7$ and $E_8$ relate 
to almost positive roots in type 
\[
T_{2m-1,2m,2m}\ ,\  T_{2m-1,2m,3m} \ ,\ T_{2m-1,2m,4m}\ .
\]
Second, colored almost positive roots in type $E_6,E_7$ and $E_8$ also relate to almost positive roots in larger 
roots system of type $A$,
in the following sense.
On the one hand, colored almost positive roots in
type $E_6,E_7$ and $E_8$ correspond to $m$-colored diagonals in a certain polygon $\Pi$, via Thm.~\ref{thm:geom_equiv}.
Such $m$-colored diagonals are obtained
by judiciously combining
two sets of oriented $m$-diagonals in $\Pi$. 
By definition oriented $m$-diagonals are a subset of all oriented diagonals in $\Pi$.
On the other hand, it is shown in \cite[\S 2.5]{L3}
that oriented diagonals in any polygon $\Pi$, together with minimal rotations among them, define a 2-repetitive cluster category of type $A$.
Generally, a $p$-repetitive cluster category of type $\Phi$ is the triangulated orbit category:
\[
\mathcal{C}_{p\Phi}=\mathcal{D}/\langle (\tau^{-1}\Sigma)^{p} \rangle,
\]
first defined in \cite{Zhu11}. 
Since indecomposable objects in $\mathcal{C}_{p\Phi}$ 
correspond to $p$-copies of almost positive roots of type $\Phi$ and the claim follows.

\begin{example}
Let $\Phi$ be a root system of type $E_6$ and let 
$\Psi$ be a root system of type $A_{10}$.
The set of two-colored almost positive roots of type $E_6$ by definition is 
\[ \Phi^2_{\geq -1}=\{-\alpha_1,-\alpha_2,\dots,-\alpha_6\}\cup\Phi_0\cup\Phi_0, \]
where $\Phi_0$ is the set of positive roots in $\Phi$.
By the above argument, the set  $\Phi^2_{\geq -1}$ is contained in four copies (two for the red oriented diagonals, two for the blue oriented diagonals) of the set 
\[
\Psi_{\geq -1}=\{-\alpha_1,-\alpha_2,\dots,-\alpha_{10}\}\cup \{\alpha_i+\cdots+\alpha_j \, \mid \,1\leq i \leq j \leq 10\}
\]
of almost positive roots of type $A_{10}:$
\[
\big\{\Psi_{\geq -1} \cup \Psi_{\geq -1}\big\}
\cup\big\{\Psi_{\geq -1} \cup \Psi_{\geq -1}\big\}.
\]
The roots in $\big\{\Psi_{\geq -1} \cup \Psi_{\geq -1}\big\}$
are in one-to-one correspondence with the indecomposable objects in a $2$-repetitive cluster category of type $A_{10}.$
\end{example}

\section{Examples of small $m$-cluster categories}\label{sec:smallex}

\subsection{The 2-cluster category of type $E_6$}
Let $m=2$ and let $\Pi^{13}$ be a regular $13$-gon 
with vertices numbered as above by the group $\mathbb Z/13\mathbb Z$.
First, for $i\in\mathbb Z/13\mathbb Z$, let $\Pi^{12}_i$
be a regular $12$-gon inside $\Pi^{13}$ with vertices numbered by $i,i+2,i+2+1,\dots, i+12$. 
A diagonal $(i,j)$ of $\Pi^{13}$ 
is called an $2$-diagonal, if $(i,j)$ is an $2$-diagonal inside 
$\Pi^{12}_i\subset \Pi^{13}.$ This convention is similar to the one
adopted in the $D_n$ case in \cite{BMd}.
Second, we double the set of all $2$-diagonals of $\Pi$
and distinguish the sets using colors, red and blue, and
subscripts $R$, $B$.
E.g.\ $(i,j)_R=(j,i)_R$ denotes a red $2$-diagonal of $\Pi^{13}$
linking the vertex $i$ to $j$.
Third, for every vertex $i$ of $\Pi^{13}$ we form
the following green
colored $2$-diagonals:
\begin{align*}\itemsep0em
(i,i+4)_G=&(i+4,i)_G:=\{ (i,i+4)_R,(i+4,i)_B\}\\
(i,i+6)_G=&(i+6,i)_G:=\{ (i,i+6)_R,(i+6,i)_B\}
\end{align*}

Let $\Gamma^{13}_{1,2,2}$ be the quiver whose vertices 
are the $2$-diagonals of $\Pi^{13}$ and an arrow between two 
vertices is drawn whenever there is a minimal clockwise rotation linking them. No arrow is drawn otherwise.

The following two theorems are obtained by computations.

\begin{theorem}
A 2-cluster tilting object in type $F_4$ consists of two blue-red arrows in one of the graph in Figure \ref{fig:Pairs of blue-red arrows in F4 2} up to rotation and two green arrows in one of the graph in Figure \ref{fig:Pairs of green arrows in F4 2} up to rotation such that the blue-red arrows do not cross green arrows. The mutation of an arrow is replacing it with another arrow such that the resulting graph satisfies the above conditions.
\end{theorem}

\begin{figure}
\includegraphics[width=0.7\textwidth]{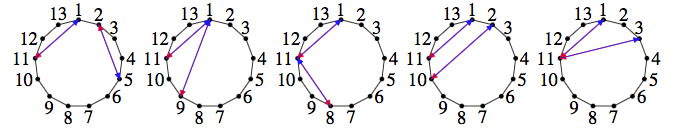}
\caption{All possible pairs of blue-red arrows appearing in a $2$-cluster tilting object in $\mathcal{C}_{F_4}^2$ up to rotation.}
\label{fig:Pairs of blue-red arrows in F4 2}
\end{figure}

\begin{figure}
\includegraphics[width=0.7\textwidth]{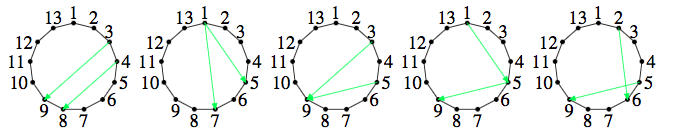}
\caption{All possible pairs of green arrows appearing in a $2$-cluster tilting object in $\mathcal{C}_{F_4}^2$ or $\mathcal{C}_{E_6}^2$ up to rotation.}
\label{fig:Pairs of green arrows in F4 2}
\end{figure}

\begin{theorem}
A $2$-cluster tilting object in type $E_6$ consists of $6$ arrows with colors red, blue, and green in a $13$-gon which satisfy the following properties.
\begin{enumerate}
\item The red arrows are of the form $[i,i+8]$ or $[i,i+10]$. The blue arrows are of the form $[i,i+3]$ or $[i,i+5]$. The green arrows are of the form $[i,i+4]$ or $[i,i+6]$.
\item If there are two or more red (resp.~blue, resp.~green) arrows, then any pair of red arrows are in one of the graphs in Figure \ref{fig:Pairs of red arrows in E6^2} (resp.~in the dual set of blue arrows with opposite orientation (not drawn), resp.~Figure \ref{fig:Pairs of green arrows in F4 2}) up to rotation.
\item The blue and red arrows do not cross green arrows.
\item Any pair of blue and red arrows is in one of the graphs in Figure \ref{fig:Pairs of blue and red arrows in E6^2} up to rotation.
\end{enumerate}
The mutation of an arrow is replacing it with another arrow such that the resulting graph satisfies the above conditions.
\end{theorem}

\begin{figure}
\includegraphics[width=0.7\textwidth]{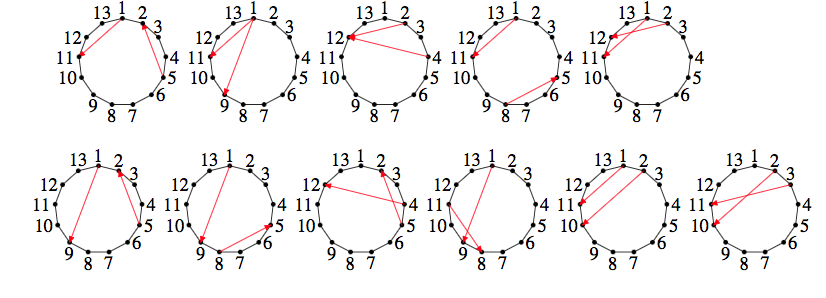}
\caption{All possible pairs of red arrows appearing in a $2$-cluster tilting object in $\mathcal{C}_{E_6}^2$ up to rotation. Dually, one finds a lists of all possible pairs of blue arrows with opposite orientation (not drawn) appearing in a $2$-cluster tilting object in $\mathcal{C}_{E_6}^2$ up to rotation.}
\label{fig:Pairs of red arrows in E6^2}
\end{figure}

\begin{figure}
\includegraphics[width=0.6\textwidth]{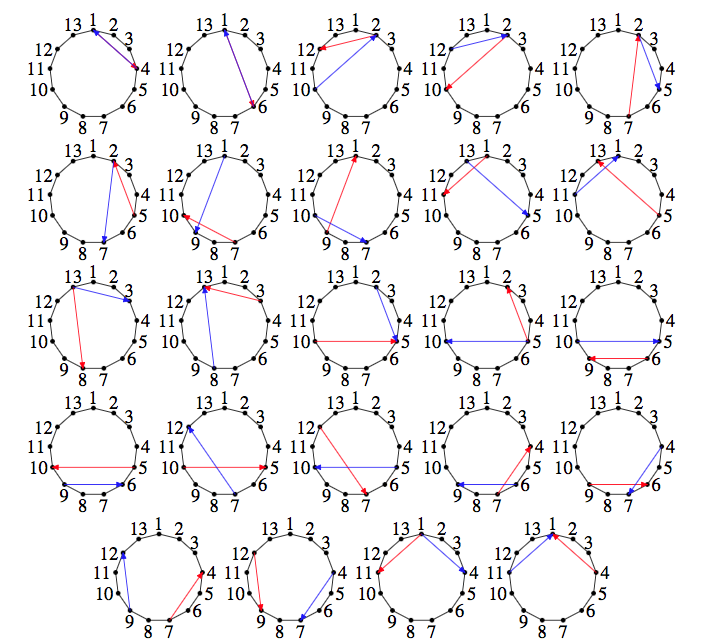}
\caption{All possible pairs of one blue and one red arrows appearing in a $2$-cluster tilting object in $\mathcal{C}_{E_6}^2$ up to rotation.}
\label{fig:Pairs of blue and red arrows in E6^2}
\end{figure}

\begin{figure}
\includegraphics[width=0.7\textwidth]{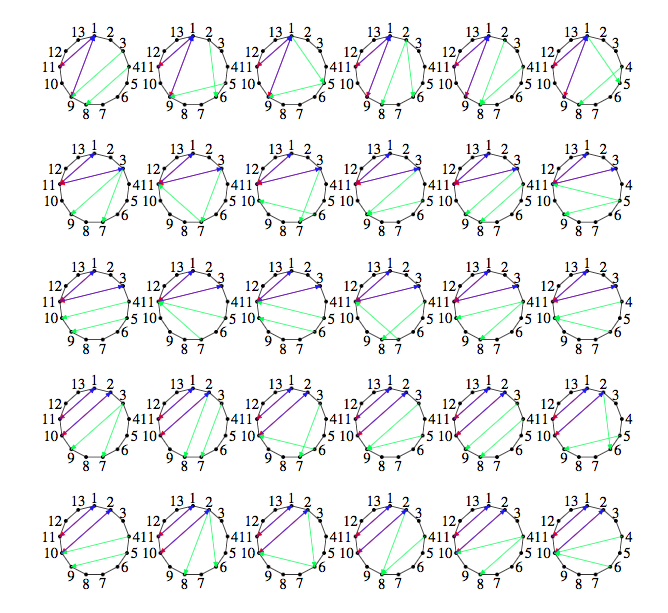}
\caption{Examples of $2$-cluster tilting objects in $\mathcal{C}_{F_4}^2$.}
\label{fig:examples of 2-clusters in F4}
\end{figure}

\begin{example}
In Figure \ref{fig:mutation of 5 11 in 2-cluster category F4}, the mutation of the arrow $[5,11]_P$ in the left cluster gives the clusters on the right.
Similarly, in Figure \ref{fig:mutation of 1 9 in 2-cluster category E6}, the mutation of the arrow $[1,9]_R$ in the left cluster gives the clusters on the right.
\end{example}
\begin{figure}
\includegraphics[width=0.2\textwidth]{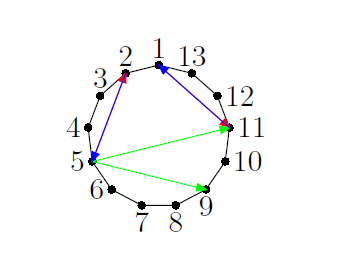}
\includegraphics[width=0.3\textwidth]{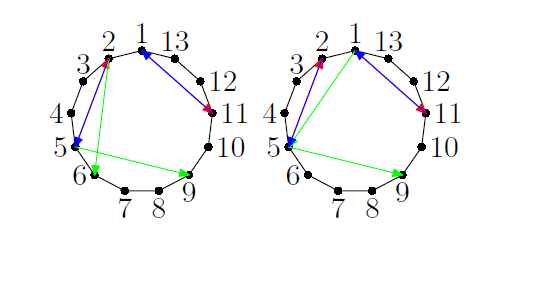}
\caption{An example of mutations in $\mathcal{C}_{F_4}^2$. }
\label{fig:mutation of 5 11 in 2-cluster category F4}
\end{figure}
\begin{figure}
\includegraphics[width=0.2\textwidth]{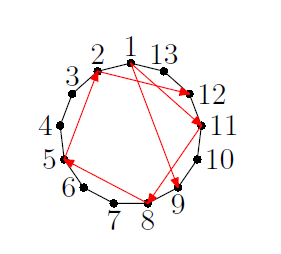}
\includegraphics[width=0.3\textwidth]{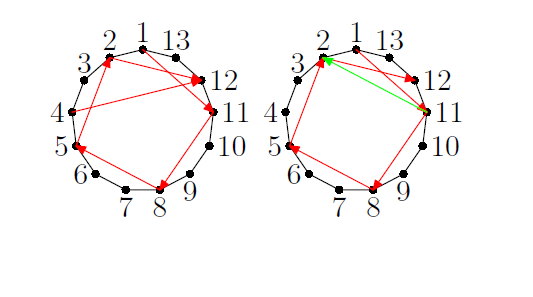}
\caption{An example of mutations in $\mathcal{C}_{E_6}^2$.}
\label{fig:mutation of 1 9 in 2-cluster category E6}
\end{figure}

\section{Cluster tilting objects in $\mathcal{C}_{F_4}$, $\mathcal{C}_{E_6}$, $\mathcal{C}_{E_7}$, $\mathcal{C}_{E_8}$}\label{sec:ct}
In this section we provide new descriptions of all cluster 
tilting objects in $\mathcal{C}_{F_4}$, $\mathcal{C}_{E_6}$ 
extend the results obtained in \cite{L3}.
These descriptions are also new for symmetric-clusters 
and clusters in $\mathbb C[{\rm Gr}(3,7)]$, 
via \cite[Prop.4.1]{BMRRT} together with \cite[Thm.5]{Sco}.

\subsection{Cluster tilting objects in $\mathcal{C}_{F_4}$ and $\mathcal{C}_{E_6}$}
Theorem \ref{thm:1} and Theorem \ref{thm:2}
are obtained by direct computations.
\begin{thm}\label{thm:1}
A cluster tilting object in type $F_4$ consists of two blue-red arrows, as in Figure \ref{fig: cluster tilting blue-red arrows} up to rotation, and two green arrows, as in Figure \ref{fig: cluster tilting green arrows} up to rotation. The arrows always combine, in such a way that the blue-red arrows do not cross green edges. However, blue-red arrows and green arrows can overlap.

The mutation of a blue-red arrow is to replace it with a blue-red arrow which satisfies the above conditions. The mutation of a green arrow is to replace it with a green arrow which satisfies the above conditions. 
\end{thm}

\begin{figure}
\includegraphics[width=0.5\textwidth]{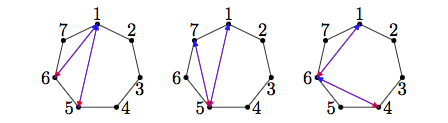}
\caption{All possible pairs of blue-red arrows in a cluster 
tiling object in $\mathcal{C}_{F_4}$ up to rotation.}
\label{fig: cluster tilting blue-red arrows}
\end{figure}

\begin{figure}
\includegraphics[width=0.5\textwidth]{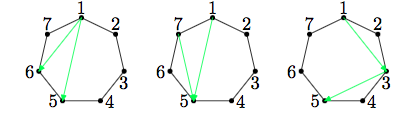}
\caption{All possible pairs of green arrows in a cluster tiling object in $\mathcal{C}_{F_4}$ or $\mathcal{C}_{E_6}$ up to rotation.}
\label{fig: cluster tilting green arrows}
\end{figure}

Figure \ref{fig:all cluster tilting object in F4 up to rotation} are all cluster tilting objects (up to rotation) in type $F_4$. 

\begin{figure}
\includegraphics[width=0.8\textwidth]{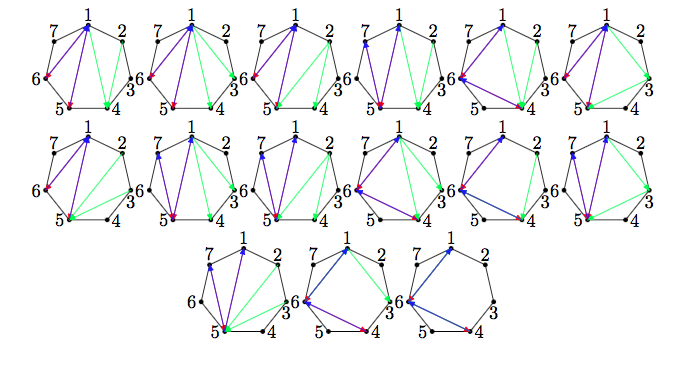}
\caption{All cluster tilting objects in $\mathcal{C}_{F_4}$ up to rotation. All clusters have four diagonals, but some are overlapped.}
\label{fig:all cluster tilting object in F4 up to rotation}
\end{figure}

\begin{remark}
In Figure \ref{fig:all cluster tilting object in F4 up to rotation}, the diagonals $[1,6]$ and $[4,6]$ in the last cluster have both a green arrow and a blue-red arrow. The diagonal $[1,6]$ in the second cluster in the third line has both a green arrow and a blue-red arrow.
\end{remark}

\begin{thm}\label{thm:2}
A cluster tilting object in type $E_6$ consists of $6$ arrows with colors red, blue, and green in a heptagon (the vertices are labelled clockwise) which satisfy the following properties.

\begin{enumerate}
\item The red edges are of the form $[i,i+4]_R$ or $[i,i+5]_R$. The blue edges are of the form $[i,i+2]_B$ or $[i,i+3]_B$. The green edges are of the form $[i,i+2]_P$ or $[i,i+3]_P$. 

\item If there are two or more red (resp.~blue, green) arrows, then any pair of the red arrows is in one of the graph in Figure \ref{fig: E6 cluster tilting red arrows} (resp.~in the
dual set of blue arrows with opposite orientation,  resp.~Figure \ref{fig: cluster tilting green arrows}) up to rotation.

\item The blue and red arrows do not cross green arrows. A blue arrow can overlap with a green arrow and a red arrow can be the opposite arrow of a green arrow. 

\item Any pair of blue and red arrows is in one of the graphs in Figure \ref{fig: E6 cluster tilting blue and red arrows} up to rotation.
\end{enumerate}

A mutation of an arrow in a cluster is replacing the diagonal with another arrow such that the resulting graph satisfies all the above conditions. 
\end{thm}

\begin{figure}
\includegraphics[width=0.7\textwidth]{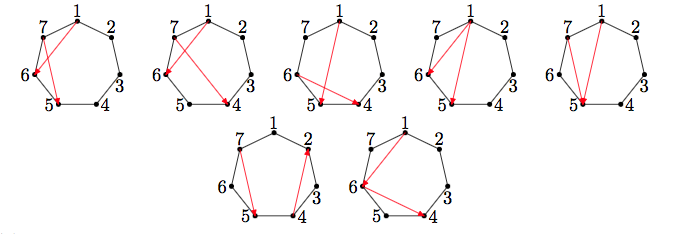}
\caption{All possible pairs of red arrows in a cluster tiling object in $\mathcal{C}_{E_6}$ up to rotation.
Dually, one finds a lists of all possible pairs of blue arrows with opposite orientation (not drawn) appearing in a cluster tilting object in $\mathcal{C}_{E_6}$ up to rotation.}
\label{fig: E6 cluster tilting red arrows}
\end{figure}

\begin{figure}
\includegraphics[width=0.8\textwidth]{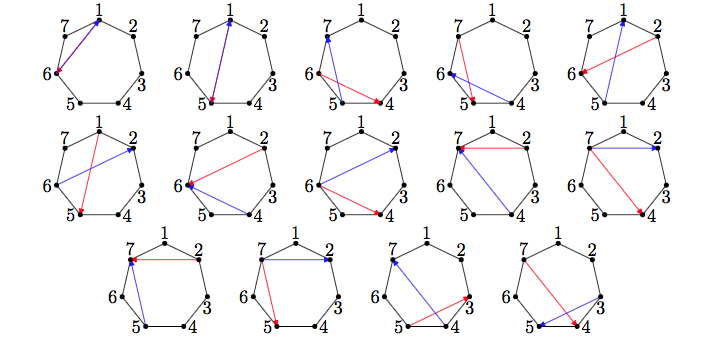}
\caption{All possible pairs of one blue and one red arrows in a cluster tiling object in $\mathcal{C}_{E_6}$ up to rotation.}
\label{fig: E6 cluster tilting blue and red arrows}
\end{figure}

In Figure \ref{fig:examples of clusters E6} thirty examples of cluster tilting objects in $\mathcal{C}_{E_6}$, are given. 
More cluster tilting objects can be 
deduced from them via $\frac{\pi}{7}$-rotations. 

\begin{figure}
\includegraphics[width=0.7\textwidth]{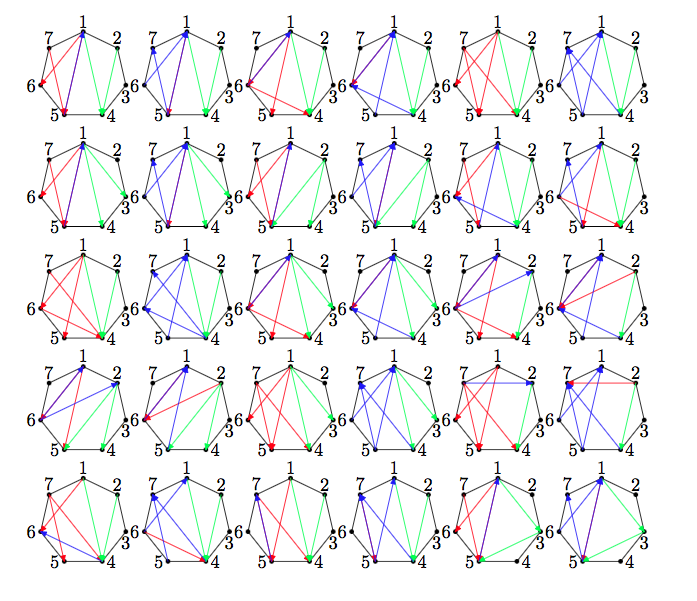}
\caption{Examples of clusters tilting objects in $\mathcal{C}_{E_6}$.}
\label{fig:examples of clusters E6}
\end{figure}

\subsection{Cluster tilting objects in $\mathcal{C}_{E_7}$ and $\mathcal{C}_{E_8}$}

The cases of types $E_7$, $E_8$ are complicated and only 
some examples of cluster tilting objects and their mutations will be provided.

\begin{example}
In Figure \ref{fig:mutation of 9 5 in cluster category E7}, the mutation of the arrow $[9,5]_R$ in the left cluster gives the cluster on the right hand side. In Figure \ref{fig:mutation of 3 5 in cluster category E7}, the mutation of the arrow $[3,5]_R$ in the left cluster gives the cluster on the right hand side.

Similarly, in Figure \ref{fig:mutation of 9 13 in cluster category E8}, 
the mutation of the arrow $[9,13]_R$ in the left cluster gives the cluster on the right hand side. In Figure \ref{fig:mutation of 10 12 in cluster category E8}, 
the mutation of the arrow $[10,12]_R$ in the left cluster gives the cluster on the right hand side.
\end{example}

\begin{figure}
\begin{minipage}{.45\textwidth}
\includegraphics[width=1.1\textwidth]{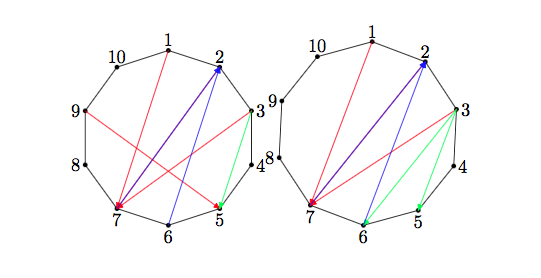}
\caption{A mutation in $\mathcal{C}_{E_7}$. }
\label{fig:mutation of 9 5 in cluster category E7}
\end{minipage}
\begin{minipage}{.45\textwidth}
\includegraphics[width=1.1\textwidth]{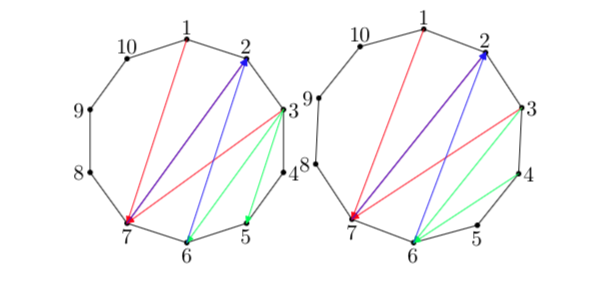}
\caption{A mutation in $\mathcal{C}_{E_7}$. }
\label{fig:mutation of 3 5 in cluster category E7}
\end{minipage}
\end{figure}

\begin{figure}
\begin{minipage}{.45\textwidth}
\includegraphics[width=1.1\textwidth]{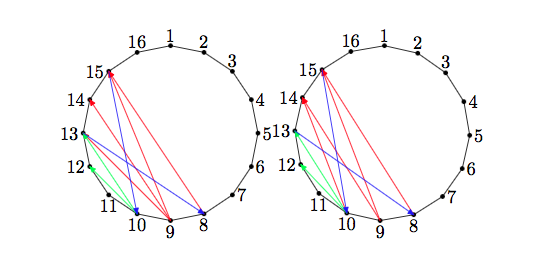}
\caption{A mutation in $\mathcal{C}_{E_8}$. }
\label{fig:mutation of 9 13 in cluster category E8}
\end{minipage}
\begin{minipage}{.45\textwidth}
\includegraphics[width=1.1\textwidth]{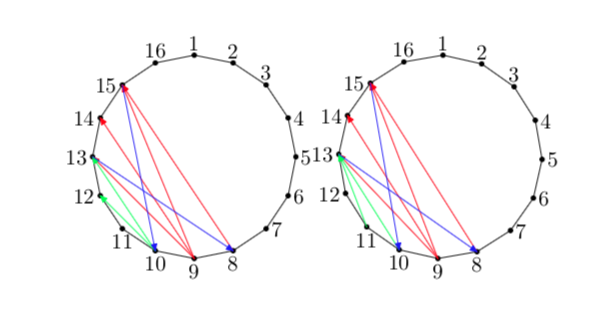}
\caption{A mutation in $\mathcal{C}_{E_8}$. }
\label{fig:mutation of 10 12 in cluster category E8}
\end{minipage}
\end{figure}

\section{Semi-standard Young tableaux and colored diagonals} \label{sec:tableaux and colored oriented diagonals}

In this section we link semi-standard Young tableaux to indecomposable objects of cluster categories of type $E_6,E_7$ and $E_8$. To do so, we exploit the fact that indecomposable objects in these categories 
are in 1-1 correspondence with almost positive roots in the root system of the same type
\cite[Prop.~4.1]{BMRRT} and hence with cluster variables in cluster algebras of the same type \cite{FZ}.  We then use the fact that every cluster variable in $\mathbb{C}[{\rm Gr}(n,m)]$ corresponds to some tableau $T \in {\rm SSYT}(n, [m])$, as described in 
\cite{CDFL}.

\subsection{Tableaux and colored diagonals in $\mathcal{C}_{E_6}$}

\begin{de}
For $T \in {\rm SSYT}(n, [m])$ which corresponds to a cluster variable in 
the ring $\mathbb C[\mathrm{Gr}(n,m)],$ we define $\tau(T)$ to be the tableau in $ {\rm SSYT}(n, [m])$  such that $\tau({\rm ch}(T))$ is equal to ${\rm ch}(\tau(T))$ (possibly multiply by some frozen variables). 
\end{de}

There are $28$ rank $1$ cluster variables and $14$ rank $2$ cluster variables in $\mathcal{C}_{E_6}$. 

In \cite{Sco}, Scott studied the correspondence between almost positive roots and cluster variables in type $E_6$, $E_8$ explicitly. The cluster 
\[
{\bf x} = \left( \left(\begin{array}{c} 3\\ 4\\ 6 \end{array}\right), \left(\begin{array}{c} 2\\ 4\\ 6 \end{array}\right), \left(\begin{array}{cc} 1&2\\ 4&5\\ 6&7 \end{array}\right), \left(\begin{array}{cc} 1&2\\ 3&5\\ 6&7 \end{array}\right), \left(\begin{array}{c} 1\\ 4\\ 7 \end{array}\right),  \left(\begin{array}{c} 2\\ 3\\ 7 \end{array}\right) \right)
\]
corresponds to the negative simple roots $(-\alpha_6, -\alpha_5, -\alpha_4, -\alpha_3, -\alpha_2,  -\alpha_1)$. We identify the cluster $([1,6]_R, [1,5]_R, [2,5]_G, [5,1]_B, [2,4]_G,  [6,1]_B)$ with ${\bf x}$. 
In this way, the following two maps agree:
\[
\tau \left(\begin{array}{c} 3\\ 4\\ 6 \end{array}\right) = \left(\begin{array}{c} 2\\ 4\\ 5 \end{array}\right)
\longleftrightarrow \tau([1,6]_R) = \rho([7,5]_R) =[5,7]_B.
 \]
In this way, one can 
associate to every colored diagonal a tableaux, by repeatedly applying the $\tau$-twist.
Below, we describe this correspondence explicitly,  see also
Figures \ref{fig: AR quiver E6} and \ref{fig: AR quiver E6 tableau}. 
We denote by $T_D$ the tableau corresponding to a diagonal $D$. 
We also write $D=T_D$.

For red diagonals, we have 
\begin{align*}
& [1,6]_R = \left(\begin{array}{c} 3\\ 4\\ 6 \end{array}\right), [2,7]_R = \left(\begin{array}{c} 1\\ 5\\ 7 \end{array}\right), [3,1]_R = \left(\begin{array}{c} 2\\ 6\\ 7 \end{array}\right), [4,2]_R = \left(\begin{array}{c} 1\\ 3\\ 4 \end{array}\right), [5,3]_R = \left(\begin{array}{c} 2\\ 3\\ 5 \end{array}\right), \\
& \scalemath{0.97}{ [6,4]_R = \left(\begin{array}{c} 4\\ 6\\ 7 \end{array}\right), [7,5]_R = \left(\begin{array}{c} 1\\ 5\\ 6 \end{array}\right), [1,5]_R = \left(\begin{array}{c} 2\\ 4\\ 6 \end{array}\right), [2,6]_R = \left(\begin{array}{cc} 1 & 3\\ 4 & 5\\ 6& 7 \end{array}\right),[3,7]_R = \left(\begin{array}{c} 2\\ 5\\ 7 \end{array}\right), } \\
& [4,1]_R = \left(\begin{array}{cc} 1 & 3\\ 2 & 6\\ 4& 7 \end{array}\right),[5,2]_R = \left(\begin{array}{c} 1\\ 3\\ 5 \end{array}\right), [6,3]_R = \left(\begin{array}{cc} 2 & 4\\ 3 & 6\\ 5& 7 \end{array}\right),[7,4]_R = \left(\begin{array}{c} 1\\ 4\\ 6 \end{array}\right).
\end{align*}

For blue diagonals, we have
\begin{align*}
& [6,1]_B = \left(\begin{array}{cc} 2 \\ 3 \\ 7 \end{array}\right), [7,2]_B = \left(\begin{array}{c} 1\\ 2\\ 4 \end{array}\right), [1,3]_B = \left(\begin{array}{c} 3\\ 5\\ 6 \end{array}\right), [2,4]_B = \left(\begin{array}{c} 4\\ 5\\ 7 \end{array}\right), [3,5]_B = \left(\begin{array}{c} 1\\ 2\\ 6 \end{array}\right), \\
& \scalemath{0.93}{ [4,6]_B = \left(\begin{array}{c} 1\\ 3\\ 7 \end{array}\right), [5,7]_B = \left(\begin{array}{c} 2\\ 4\\ 5 \end{array}\right), [5,1]_B = \left(\begin{array}{cc} 1&2\\ 3&5\\ 6&7 \end{array}\right), [6,2]_B=\left(\begin{array}{c} 2\\ 4\\ 7 \end{array}\right), [7,3]_B = \left(\begin{array}{cc} 1 & 3\\ 2 & 5\\ 4& 6 \end{array}\right), } \\
& [1,4]_B = \left(\begin{array}{c} 3\\ 5\\ 7 \end{array}\right), [2,5]_B = \left(\begin{array}{cc} 1 & 4\\ 2 & 6\\ 5& 7 \end{array}\right),[3,6]_B = \left(\begin{array}{c} 1\\ 3\\ 6 \end{array}\right), [4,7]_B = \left(\begin{array}{cc} 1 & 2\\ 3 & 4\\ 5& 7 \end{array}\right).
\end{align*}

For a sequence of numbers $i_1 < \ldots < i_m$, we say that $i_j$ ($j \in [m]$) is the predecessor of $i_{j+1}$ and $i_{j+1}$ is the successor of $i_j$, where we use the convention that $i_{m+1}=i_1$. Denote by $s(T)$ the sequence (from small to large) consisting of entries of a tableau $T$. 

For green diagonals, we have the following. For $i \in \{1, \ldots, 7\}$, $[i,i+2]_G$ corresponds to the one-column tableau with entries $\{a,b,c\}$, where $a$ is the common element in both $[i+2,i]_R$ and $[i,i+2]_B$, $b$ is the successor of $a$ in $s(T_{[i+2,i]_R})$, and $c$ is the predecessor of $a$ in $s(T_{[i,i+2]_B})$. For example, $[2,4]_G = \left(\begin{array}{c} 1\\ 4\\ 7 \end{array}\right)$.  For $i \in [7]$, exactly one of $[i+3,i]_R$ and $[i,i+3]_B$ is a two column tableaux. Denote the tableau by $\left(\begin{array}{cc} i_1 & j_1\\ i_2 & j_2\\ i_3& j_3 \end{array}\right)$. Then $[i,i+3]_G =  \left(\begin{array}{cc} i_1 & i_2\\ j_1 & i_3\\ j_2& j_3 \end{array}\right)$.

A set of tableaux are called compatible if they are in the same cluster. The geometric model of cluster categories can be used to obtain results about the compatibility of tableaux. For example, in the first graph in Figure \ref{fig:examples of clusters E6}, the diagonals $[1,6]_R$, $[1,5]_R$, $[7,5]_R$, $[5,1]_B$, $[1,4]_G$, $[2,4]_G$ form a cluster. Therefore the corresponding tableaux
\begin{align*}
\left(\begin{array}{c} 3\\ 4\\ 6 \end{array}\right), \left(\begin{array}{c} 2\\ 4\\ 6 \end{array}\right),
\left(\begin{array}{c} 1\\ 5\\ 6 \end{array}\right),
\left(\begin{array}{cc} 1&2\\ 3&5\\ 6&7 \end{array}\right),
\left(\begin{array}{cc} 1&2\\ 3&4\\ 6&7 \end{array}\right),
\left(\begin{array}{c} 1\\ 4\\ 7 \end{array}\right),
\end{align*}
are compatible.

\subsection{Tableaux and colored diagonals in $\mathcal{C}_{E_7}$}

In order to have a correspondence between oriented diagonals and tableaux in the case of $E_7$, we froze $P_{167}$ in the initial quiver in Figure \ref{fig:initial cluster for E7}. Note that the Pl\"{u}cker coordinate $P_{i_1, \ldots, i_k}$ corresponds to the one column tableau with entries $i_1, \ldots, i_k$. 

\begin{figure}
\resizebox{0.25\textwidth}{!}{%
\begin{tikzpicture}[scale=0.5]
     \node at (0,0) (v00) {\fbox{$P_{123}$}};
     \node at (0,-4) (v10) {$P_{124}$};
     \node at (0,-8) (v20) {$P_{125}$};
     \node at (0,-12) (v30) {$P_{126}$};
     \node at (0,-16) (v40) {$P_{127}$};
     \node at (0,-20) (v50) {\fbox{$P_{128}$}};

     \node at (4,-4) (v11) {$P_{134}$};
     \node at (4,-8) (v21) {$P_{145}$};
     \node at (4,-12) (v31) {$P_{156}$};
     \node at (4,-16) (v41) {\fbox{$P_{167}$}};
     \node at (4,-20) (v51) {\fbox{$P_{178}$}};
     
     \node at (8,-4) (v12) {\fbox{$P_{234}$}};
     \node at (8,-8) (v22) {\fbox{$P_{345}$}};
     \node at (8,-12) (v32) {\fbox{$P_{456}$}};
     \node at (8,-16) (v42) {\fbox{$P_{567}$}};
     \node at (8,-20) (v52) {\fbox{$P_{678}$}};
     
     \draw[->] (v10)--(v00);
     \draw[->] (v20)--(v10);
     \draw[->] (v30)--(v20);
     \draw[->] (v40)--(v30);
     \draw[->] (v50)--(v40);

     \draw[->] (v21)--(v11);
     \draw[->] (v31)--(v21);
     \draw[->] (v41)--(v31);
     \draw[->] (v51)--(v41);
     
     \draw[->] (v22)--(v12);
     \draw[->] (v32)--(v22);
     \draw[->] (v42)--(v32);
     \draw[->] (v52)--(v42);

     \draw[->] (v11)--(v10);
     \draw[->] (v12)--(v11);

     \draw[->] (v21)--(v20);
     \draw[->] (v22)--(v21);

     \draw[->] (v31)--(v30);
     \draw[->] (v32)--(v31);

     \draw[->] (v41)--(v40);
     \draw[->] (v42)--(v41);
    
     \draw[->] (v10)--(v21);
     \draw[->] (v21)--(v32);

     \draw[->] (v20)--(v31);
     \draw[->] (v31)--(v42);

     \draw[->] (v30)--(v41);
     \draw[->] (v41)--(v52);
     
     \draw[->] (v40)--(v51);
     
     \draw[->] (v11)--(v22);
\end{tikzpicture} 
}
            \caption{The initial cluster for $\mathcal{C}_{E_7}$.}
            \label{fig:initial cluster for E7}
\end{figure}
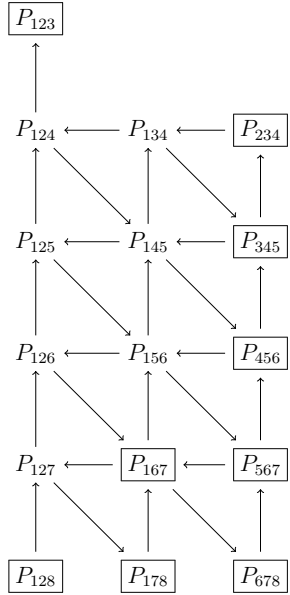

There are $33$ rank $1$ cluster variables, $29$ rank $2$ cluster variables, and $8$ rank $3$ cluster variables in $\mathcal{C}_{E_7}$. They can be obtained from the initial quiver in Figure \ref{fig:initial cluster for E7} by mutations.

We identify the cluster $([10,6]_R, [1,6]_R, [1,5]_R, [2,5]_G,  [5,1]_B, [2,4]_G, [6,1]_B)$ with the cluster
\begin{align*}
{\bf x} = \left( \left(\begin{array}{c} 1\\ 5\\ 7 \end{array}\right), \left(\begin{array}{cc} 1&2\\ 3&5\\ 7&8 \end{array}\right), \left(\begin{array}{ccc} 1&2&3\\ 3&4&5\\ 6&7&8 \end{array}\right), \left(\begin{array}{ccc} 1&2&3\\ 4&4&5\\ 6&7&8 \end{array}\right), \left(\begin{array}{cc} 1&2\\ 4&5\\ 6&8 \end{array}\right), \left(\begin{array}{c} 2\\ 4\\ 7 \end{array}\right), \left(\begin{array}{c} 2\\ 5\\ 6 \end{array}\right) \right).
\end{align*}

The correspondence between tableaux and diagonals 
can be deduced by combining the Auslander--Reiten
quiver of $\mathcal{C}_{E_7}$
with the stable translations graph given in Figure \ref{fig: AR quiver E7 tableau part 1} and Figure \ref{fig: AR quiver E7 tableau part 2}. The first two slices from the
Auslander--Reiten quiver of $\mathcal{C}_{E_7}$
are illustrated on the left in Figure \ \ref{fig: AR quiver E7}.

\subsection{Tableaux and colored diagonals in $\mathcal{C}_{E_8}$}

There are $48$ rank $1$ cluster variables, $56$ rank $2$ cluster variables, and $24$ rank $3$ cluster variables in $\mathcal{C}_{E_8}$. 

In \cite{Sco}, Scott verified that in type $E_8$, the cluster 
\[
\scalemath{0.9}{ {\bf x} = \left(\left(\begin{array}{c} 4\\ 5\\ 7 \end{array}\right), \left(\begin{array}{c} 3\\ 5\\ 7 \end{array}\right), \left(\begin{array}{cc} 1&3\\ 5&6\\ 7&8 \end{array}\right), \left(\begin{array}{ccc} 1&1&3\\ 2&5&6\\ 4&7&8 \end{array}\right), \left(\begin{array}{ccc} 1&2&3\\ 2&5&6\\ 4&7&8 \end{array}\right), \left(\begin{array}{cc} 2&3\\ 4&6\\ 7&8 \end{array}\right), \left(\begin{array}{c} 2\\ 5\\ 8 \end{array}\right),  \left(\begin{array}{c} 3\\ 4\\ 8 \end{array}\right) \right) }
\]
correspond to the negative simple roots $(-\alpha_8, - \alpha_7, -\alpha_6, -\alpha_5, -\alpha_4, -\alpha_3, -\alpha_2, -\alpha_1)$. We identify the cluster $([16,7]_R, [16,6]_R, [1,6]_R, [1,5]_R, [2,5]_G,  [5,1]_B, [2,4]_G, [6,1]_B)$ with ${\bf x}$. 

The correspondence between tableaux and diagonals can be deduced by combining the Auslander--Reiten
quiver of $\mathcal{C}_{E_8}$, with initial slices given as in Figure \ref{fig: AR quiver E8 part 1},
with the stable translations graph given in Figure \ref{fig: AR quiver E8 part 1 tableaux}, Figure \ref{fig: AR quiver E8 part 2 tableaux} and Figure \ref{fig: AR quiver E8 part 3 tableaux}.

\begin{proposition}
In the Auslander--Reiten quivers of $\mathcal{C}_{E_6}$, $\mathcal{C}_{E_7}$, $\mathcal{C}_{E_8}$, for
every mesh of the form:
\[
\xymatrix@-6mm@C-0.2cm{
& & T_1 \ar[rdd]\\
 \\
&  S_1  \ar[r]     \ar[rdd]\ar[ruu]&  \ar[r] \vdots   & S_2  \\
 \\
& & T_k \ar[ruu]  \\
} 
\]
where $k \in \{2,3\}$, we have that $S_1 \cup S_2 = T_1 \cup \cdots \cup T_k$ up to permutation of entries in $S_1 \cup S_2$. 
\end{proposition}

\begin{proof}
The proposition follows from the explicit descriptions of the 
Auslander--Reiten quivers of $\mathcal{C}_{E_6}$, $\mathcal{C}_{E_7}$, $\mathcal{C}_{E_8}$.
\end{proof}

\begin{figure}[h]
$$\scalemath{0.5}{
\xymatrix@-6mm@C-0.2cm{
& & \text{$\left(\begin{array}{c} 3\\ 4\\ 6 \end{array}\right)$}   \ar[rdd] & &    \text{$\left(\begin{array}{c} 1\\ 5\\ 7 \end{array}\right)$}    \ar[rdd] & &   \text{$\left(\begin{array}{c} 2\\ 6\\ 7 \end{array}\right)$}   \ar[rdd] &&  \text{$\left(\begin{array}{c} 1\\ 3\\ 4 \end{array}\right)$}     \ar[rdd]&&  \text{$\left(\begin{array}{c} 2\\ 3\\ 5 \end{array}\right)$}   \ar[rdd] &&  \text{$\left(\begin{array}{c} 4\\ 6\\ 7 \end{array}\right)$}   \ar[rdd] &&  \text{$\left(\begin{array}{c} 1\\ 5\\ 6 \end{array}\right)$}   \ar[rdd] &&  \text{$\left(\begin{array}{c} 2\\ 3\\ 7 \end{array}\right)$}   \\
\\
&  \text{$\left(\begin{array}{c} 2\\ 4\\ 6 \end{array}\right)$}   \ar[rdd] \ar[ruu] & &  \text{$\left(\begin{array}{cc} 1&3\\ 4&5\\ 6&7 \end{array}\right)$} \ar[rdd]\ar[ruu] & &    \text{$\left(\begin{array}{c} 2\\ 5\\ 7 \end{array}\right)$} \ar[rdd] \ar[ruu] & &   \text{$\left(\begin{array}{cc} 1&3\\ 2&6\\ 4&7 \end{array}\right)$}   \ar[ruu]\ar[rdd]&&  \text{$\left(\begin{array}{c} 1\\ 3\\ 5\end{array}\right)$} \ar[ruu]\ar[rdd]&&   \text{$\left(\begin{array}{cc} 2&4\\ 3&6\\ 5&7 \end{array}\right)$} \ar[rdd] \ar[ruu]  &&  \text{$\left(\begin{array}{c} 1\\ 4\\ 6\end{array}\right)$} \ar[ruu]\ar[rdd] &&  \text{$\left(\begin{array}{cc} 1&2\\ 3&5\\ 6&7\end{array}\right)$} \ar[ruu]\ar[rdd]   \\
 \\
&  \ar[r]  \text{$\left(\begin{array}{c} 1\\ 4\\ 7 \end{array}\right)$}    & \ar[r]  \text{$\left(\begin{array}{cc} 1&2\\ 4&5\\ 6&7 \end{array}\right)$} \ar[ruu]\ar[rdd] &  \ar[r]   \text{$\left(\begin{array}{c} 2\\ 5\\ 6 \end{array}\right)$}   &\ar[r] \ar[ruu] \ar[rdd] \text{$\left(\begin{array}{cc} 2&3\\ 4&5\\ 6&7 \end{array}\right)$} &  \ar[r]   \text{$\left(\begin{array}{c} 3\\ 4\\ 7 \end{array}\right)$}  &\ar[r]    \text{$\left(\begin{array}{cc} 1&3\\ 2&5\\ 4&7 \end{array}\right)$}  \ar[ruu] \ar[rdd]&  \ar[r]  \text{$\left(\begin{array}{c} 1\\ 2\\ 5 \end{array}\right)$}   & \ar[r]  \text{$\left(\begin{array}{cc} 1&3\\ 2&6\\ 5&7 \end{array}\right)$}   \ar[rdd]\ar[ruu]&  \ar[r]  \text{$\left(\begin{array}{c} 3\\ 6\\ 7 \end{array}\right)$} & \ar[r]   \text{$\left(\begin{array}{cc} 1&4\\ 3&6\\ 5&7 \end{array}\right)$}  \ar[rdd]\ar[ruu]& \text{$\left(\begin{array}{c} 1\\ 4\\ 5 \end{array}\right)$} \ar[r] & \ar[r]   \text{$\left(\begin{array}{cc} 1&2\\ 3&4\\ 5&6 \end{array}\right)$}  \ar[rdd]\ar[ruu]& \text{$\left(\begin{array}{c} 2\\ 3\\ 6 \end{array}\right)$} \ar[r] & \ar[r]   \text{$\left(\begin{array}{cc} 1&2\\ 3&4\\ 6&7 \end{array}\right)$}  \ar[rdd]\ar[ruu]& \text{$\left(\begin{array}{c} 1\\ 4\\ 7 \end{array}\right)$} \ar[r] &  \text{$\left(\begin{array}{cc} 1&2\\ 4&5\\ 6&7 \end{array}\right)$}  \\
 \\
& \text{$\left(\begin{array}{cc} 1&2\\ 3&5\\ 6&7 \end{array}\right)$}   \ar[ruu] \ar[rdd]&&   \text{$\left(\begin{array}{c} 2\\ 4\\ 7 \end{array}\right)$}  \ar[ruu] \ar[rdd]&&  \text{$\left(\begin{array}{cc} 1&3\\ 2&5\\ 4&6 \end{array}\right)$}    \ar[ruu] \ar[rdd]&&   \text{$\left(\begin{array}{c} 3\\ 5\\ 7 \end{array}\right)$}  \ar[ruu] \ar[rdd]&&   \text{$\left(\begin{array}{cc} 1&4\\ 2&6\\ 5&7 \end{array}\right)$} \ar[ruu] \ar[rdd] && \text{$\left(\begin{array}{c} 1\\ 3\\ 6 \end{array}\right)$} \ar[ruu] \ar[rdd] &&   \text{$\left(\begin{array}{cc} 1&2\\ 3&4\\ 5&7 \end{array}\right)$} \ar[ruu] \ar[rdd] && \text{$\left(\begin{array}{c} 2\\ 4\\ 6 \end{array}\right)$} \ar[ruu] \ar[rdd]  \\
 \\
&&  \text{$\left(\begin{array}{c} 2\\ 3\\ 7 \end{array}\right)$} \ar[ruu]&&  \text{$\left(\begin{array}{c} 1\\ 2\\ 4 \end{array}\right)$}   \ar[ruu]&&   \text{$\left(\begin{array}{c} 3\\ 5\\ 6 \end{array}\right)$}  \ar[ruu]&&  \text{$\left(\begin{array}{c} 4\\ 5\\ 7 \end{array}\right)$}   \ar[ruu]&&  \text{$\left(\begin{array}{c} 1\\ 2\\ 6 \end{array}\right)$} \ar[ruu] &&  \text{$\left(\begin{array}{c} 1\\ 3\\ 7 \end{array}\right)$}   \ar[ruu]&&  \text{$\left(\begin{array}{c} 2\\ 4\\ 5 \end{array}\right)$} \ar[ruu] &&  \text{$\left(\begin{array}{c} 3\\ 4\\ 6 \end{array}\right)$}  \\
} }
$$
\caption{Auslander--Reiten quiver of $\mathcal{C}_{E_6}$ with vertices labelled by tableaux.}
\label{fig: AR quiver E6 tableau}
\end{figure}
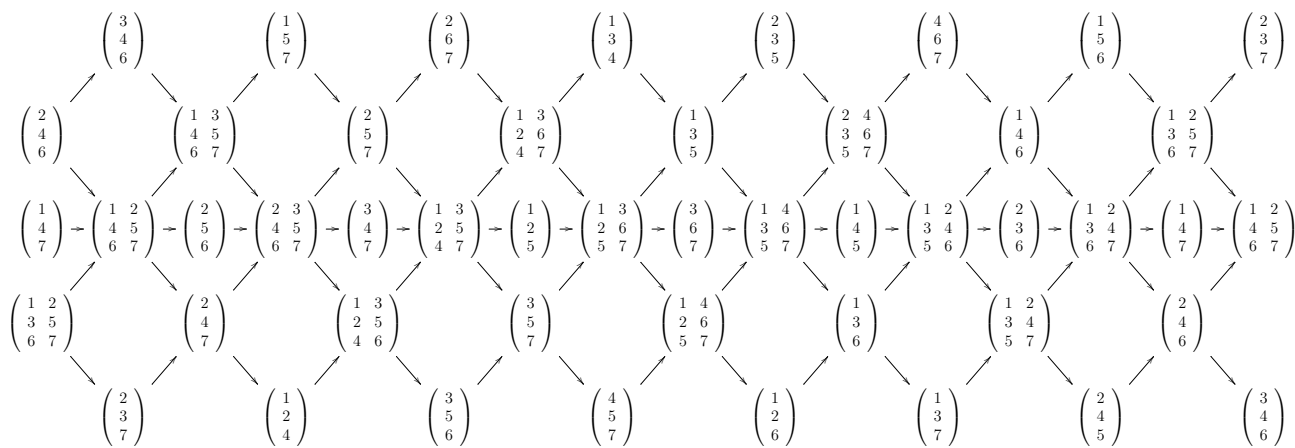

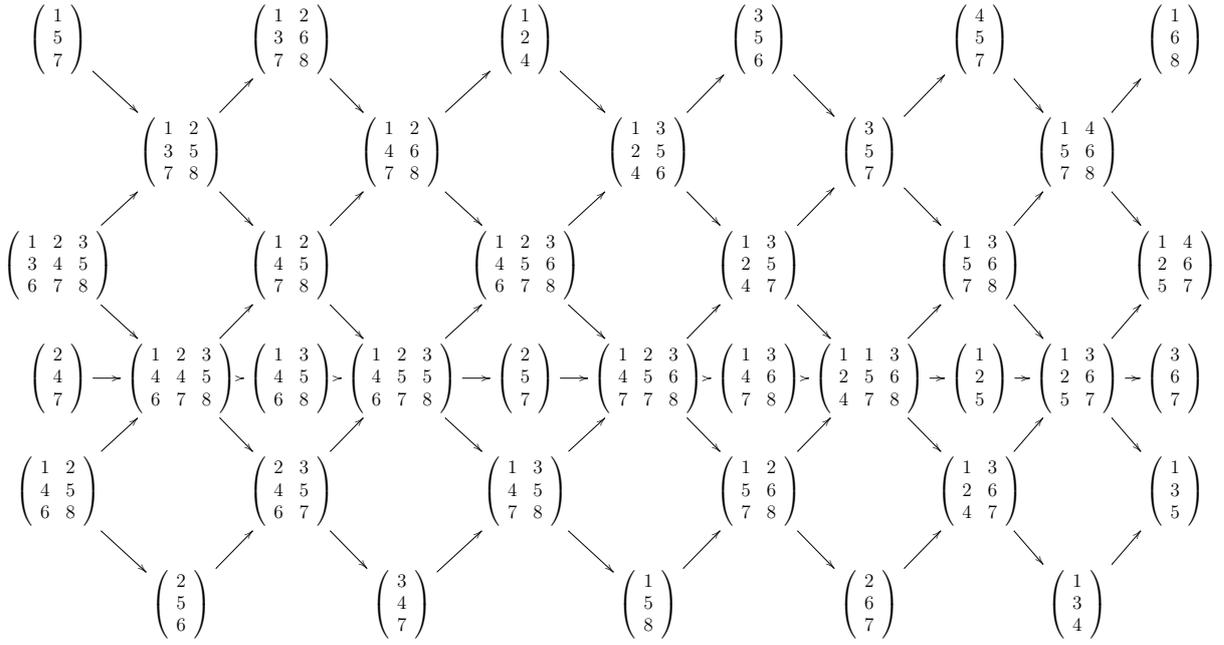
\begin{figure}
$$\scalemath{0.6}{
\xymatrix@-6mm@C-0.2cm{
& \text{$\left(\begin{array}{c} 1\\ 5\\ 7 \end{array}\right)$} \ar[rdd] & & \text{$\left(\begin{array}{cc} 1&2\\ 3&6\\ 7&8 \end{array}\right)$} \ar[rdd] & &    \text{$\left(\begin{array}{c} 1\\ 2\\ 4 \end{array}\right)$}   \ar[rdd] & &   \text{$\left(\begin{array}{c} 3\\ 5\\ 6 \end{array}\right)$}  \ar[rdd] &&  \text{$\left(\begin{array}{c} 4\\ 5\\ 7 \end{array}\right)$}   \ar[rdd]&&  \text{$\left(\begin{array}{c} 1\\ 6\\ 8 \end{array}\right)$} \\
\\
& & \text{$\left(\begin{array}{cc} 1&2\\ 3&5\\ 7&8 \end{array}\right)$}  \ar[ruu] \ar[rdd] & &    \text{$\left(\begin{array}{cc} 1&2\\ 4&6\\ 7&8 \end{array}\right)$}  \ar[ruu]   \ar[rdd] & &   \text{$\left(\begin{array}{cc} 1&3\\ 2&5\\ 4&6 \end{array}\right)$}  \ar[ruu]  \ar[rdd] &&  \text{$\left(\begin{array}{c} 3\\ 5\\ 7 \end{array}\right)$}   \ar[ruu]  \ar[rdd]&&  \text{$\left(\begin{array}{cc} 1&4\\ 5&6\\ 7&8 \end{array}\right)$}  \ar[ruu]   \ar[rdd] \\
\\
&  \text{$\left(\begin{array}{ccc} 1&2&3\\ 3&4&5\\ 6&7&8 \end{array}\right)$}   \ar[rdd] \ar[ruu] & &  \text{$\left(\begin{array}{cc} 1&2\\ 4&5\\ 7&8 \end{array}\right)$} \ar[rdd]\ar[ruu] & &    \text{$\left(\begin{array}{ccc} 1&2&3\\ 4&5&6\\ 6&7&8 \end{array}\right)$} \ar[rdd] \ar[ruu] & &   \text{$\left(\begin{array}{cc} 1&3\\ 2&5\\ 4&7 \end{array}\right)$}   \ar[ruu]\ar[rdd]&&  \text{$\left(\begin{array}{cc} 1&3\\ 5&6\\ 7&8 \end{array}\right)$} \ar[ruu]\ar[rdd]&&   \text{$\left(\begin{array}{cc} 1&4\\ 2&6\\ 5&7 \end{array}\right)$}    \\
 \\
&  \ar[r]  \text{$\left(\begin{array}{c} 2\\ 4\\ 7 \end{array}\right)$}    & \ar[r]  \text{$\left(\begin{array}{ccc} 1&2&3\\ 4&4&5\\ 6&7&8 \end{array}\right)$} \ar[ruu]\ar[rdd] &  \ar[r]   \text{$\left(\begin{array}{cc} 1&3\\ 4&5\\ 6&8 \end{array}\right)$}   &\ar[r] \ar[ruu] \ar[rdd] \text{$\left(\begin{array}{ccc} 1&2&3\\ 4&5&5\\ 6&7&8 \end{array}\right)$} &  \ar[r]   \text{$\left(\begin{array}{c} 2\\ 5\\ 7 \end{array}\right)$}  &\ar[r]    \text{$\left(\begin{array}{ccc} 1&2&3\\ 4&5&6\\ 7&7&8 \end{array}\right)$}  \ar[ruu] \ar[rdd]&  \ar[r]  \text{$\left(\begin{array}{cc} 1&3\\ 4&6\\ 7&8 \end{array}\right)$}   & \ar[r]  \text{$\left(\begin{array}{ccc} 1&1&3\\ 2&5&6\\ 4&7&8 \end{array}\right)$}   \ar[rdd]\ar[ruu]&  \ar[r]  \text{$\left(\begin{array}{c} 1\\ 2\\ 5 \end{array}\right)$} & \ar[r]   \text{$\left(\begin{array}{cc} 1&3\\ 2&6\\ 5&7 \end{array}\right)$}  \ar[rdd]\ar[ruu]& \text{$\left(\begin{array}{c} 3\\ 6\\ 7 \end{array}\right)$}  \\
 \\
& \text{$\left(\begin{array}{cc} 1&2\\ 4&5\\ 6&8 \end{array}\right)$}   \ar[ruu] \ar[rdd]&&   \text{$\left(\begin{array}{cc} 2&3\\ 4&5\\ 6&7 \end{array}\right)$}  \ar[ruu] \ar[rdd]&&  \text{$\left(\begin{array}{cc} 1&3\\ 4&5\\ 7&8 \end{array}\right)$}    \ar[ruu] \ar[rdd]&&   \text{$\left(\begin{array}{cc} 1&2\\ 5&6\\ 7&8 \end{array}\right)$}  \ar[ruu] \ar[rdd]&&   \text{$\left(\begin{array}{cc} 1&3\\ 2&6\\ 4&7 \end{array}\right)$} \ar[ruu] \ar[rdd] && \text{$\left(\begin{array}{c} 1\\ 3\\ 5 \end{array}\right)$}  \\
 \\
&&  \text{$\left(\begin{array}{c} 2\\ 5\\ 6 \end{array}\right)$} \ar[ruu]&&  \text{$\left(\begin{array}{c} 3\\ 4\\ 7 \end{array}\right)$}   \ar[ruu]&&   \text{$\left(\begin{array}{c} 1\\ 5\\ 8 \end{array}\right)$}  \ar[ruu]&&  \text{$\left(\begin{array}{c} 2\\ 6\\ 7 \end{array}\right)$}   \ar[ruu]&&  \text{$\left(\begin{array}{c} 1\\ 3\\ 4 \end{array}\right)$} \ar[ruu] \\
} }
$$
\caption{Auslander--Reiten quiver of $\mathcal{C}_{E_7}$ with vertices labelled by tableaux, Part 1.}
\label{fig: AR quiver E7 tableau part 1}
\end{figure}

\begin{figure}
$$\scalemath{0.6}{
\xymatrix@-6mm@C-0.2cm{
& \text{$\left(\begin{array}{c} 1\\ 6\\ 8 \end{array}\right)$} \ar[rdd] & & \text{$\left(\begin{array}{c} 1\\ 2\\ 7 \end{array}\right)$} \ar[rdd] & &    \text{$\left(\begin{array}{c} 1\\ 3\\ 8 \end{array}\right)$}   \ar[rdd] & &   \text{$\left(\begin{array}{c} 2\\ 4\\ 5 \end{array}\right)$}  \ar[rdd] &&  \text{$\left(\begin{array}{c} 3\\ 4\\ 6 \end{array}\right)$}   \ar[rdd]&&  \text{$\left(\begin{array}{c} 1\\ 5\\ 7 \end{array}\right)$} \\
\\
& & \text{$\left(\begin{array}{c} 1\\ 2\\ 6 \end{array}\right)$} \ar[ruu] \ar[rdd] & &    \text{$\left(\begin{array}{c} 1\\ 3\\ 7 \end{array}\right)$}  \ar[ruu]   \ar[rdd] & &  \text{$\left(\begin{array}{cc} 1&2\\ 3&4\\ 5&8 \end{array}\right)$}  \ar[ruu]  \ar[rdd] &&  \text{$\left(\begin{array}{c} 2\\ 4\\ 6 \end{array}\right)$}   \ar[ruu]  \ar[rdd]&&  \text{$\left(\begin{array}{cc} 1&3\\ 4&5\\ 6&7 \end{array}\right)$}  \ar[ruu]   \ar[rdd] \\
\\
&  \text{$\left(\begin{array}{cc} 1&4\\ 2&6\\ 5&7 \end{array}\right)$}   \ar[rdd] \ar[ruu] & &  \text{$\left(\begin{array}{c} 1\\ 3\\ 6 \end{array}\right)$} \ar[rdd]\ar[ruu] & &    \text{$\left(\begin{array}{cc} 1&2\\ 3&4\\ 5&7 \end{array}\right)$} \ar[rdd] \ar[ruu] & &   \text{$\left(\begin{array}{cc} 1&2\\ 3&4\\ 6&8 \end{array}\right)$}   \ar[ruu]\ar[rdd]&&  \text{$\left(\begin{array}{cc} 1&2\\ 4&5\\ 6&7 \end{array}\right)$} \ar[ruu]\ar[rdd]&&   \text{$\left(\begin{array}{ccc} 1&2&3\\ 3&4&5\\ 6&7&8 \end{array}\right)$}   \\
 \\
&  \ar[r] \text{$\left(\begin{array}{c} 3\\ 6\\ 7 \end{array}\right)$}    & \ar[r]  \text{$\left(\begin{array}{cc} 1&4\\ 3&6\\ 5&7 \end{array}\right)$} \ar[ruu]\ar[rdd] &  \ar[r]  \text{$\left(\begin{array}{c} 1\\ 4\\ 5 \end{array}\right)$}    &\ar[r] \ar[ruu] \ar[rdd] \text{$\left(\begin{array}{cc} 1&2\\ 3&4\\ 5&6 \end{array}\right)$} &  \ar[r]   \text{$\left(\begin{array}{c} 2\\ 3\\ 6 \end{array}\right)$}   &\ar[r]    \text{$\left(\begin{array}{cc} 1&2\\ 3&4\\ 6&7 \end{array}\right)$}   \ar[ruu] \ar[rdd]&  \ar[r]  \text{$\left(\begin{array}{c} 1\\ 4\\ 7 \end{array}\right)$}    & \ar[r]  \text{$\left(\begin{array}{ccc} 1&1&2\\ 3&4&5\\ 6&7&8 \end{array}\right)$}  \ar[rdd]\ar[ruu]&  \ar[r]  \text{$\left(\begin{array}{cc} 1&2\\ 3&5\\ 6&8 \end{array}\right)$} & \ar[r]   \text{$\left(\begin{array}{ccc} 1&2&2\\ 3&4&5\\ 6&7&8 \end{array}\right)$}  \ar[rdd]\ar[ruu]&  \text{$\left(\begin{array}{c} 2\\ 4\\ 7 \end{array}\right)$} \\
 \\
& \text{$\left(\begin{array}{c} 1\\ 3\\ 5 \end{array}\right)$}    \ar[ruu] \ar[rdd]&&  \text{$\left(\begin{array}{cc} 2&4\\ 3&6\\ 5&7 \end{array}\right)$}  \ar[ruu] \ar[rdd]&&  \text{$\left(\begin{array}{c} 1\\ 4\\ 6 \end{array}\right)$}    \ar[ruu] \ar[rdd]&&   \text{$\left(\begin{array}{cc} 1&2\\ 3&5\\ 6&7 \end{array}\right)$} \ar[ruu] \ar[rdd]&&   \text{$\left(\begin{array}{cc} 1&2\\ 3&4\\ 7&8 \end{array}\right)$} \ar[ruu] \ar[rdd] && \text{$\left(\begin{array}{cc} 1&2\\ 4&5\\ 6&8 \end{array}\right)$} \\
 \\
&&  \text{$\left(\begin{array}{c} 2\\ 3\\ 5 \end{array}\right)$} \ar[ruu]&&  \text{$\left(\begin{array}{c} 4\\ 6\\ 7 \end{array}\right)$}   \ar[ruu]&&   \text{$\left(\begin{array}{c} 1\\ 5\\ 6 \end{array}\right)$}  \ar[ruu]&&  \text{$\left(\begin{array}{c} 2\\ 3\\ 7 \end{array}\right)$}   \ar[ruu]&&  \text{$\left(\begin{array}{c} 1\\ 4\\ 8 \end{array}\right)$} \ar[ruu] \\
} }
$$
\caption{Auslander--Reiten quiver of $\mathcal{C}_{E_7}$ with vertices labelled by tableaux, Part 2.}
\label{fig: AR quiver E7 tableau part 2}
\end{figure}
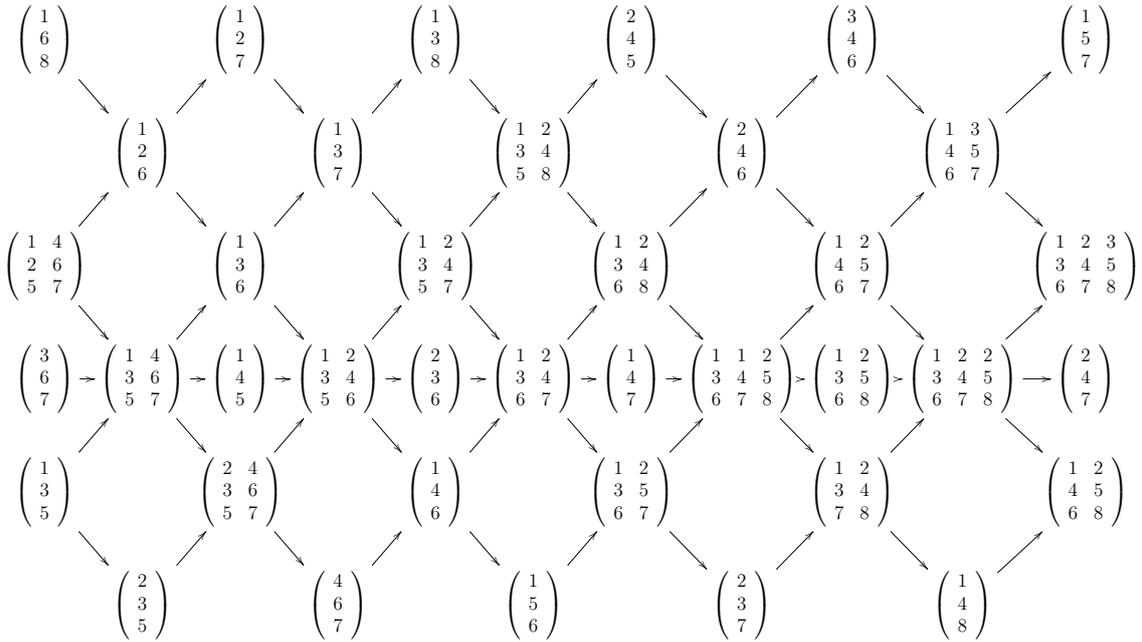

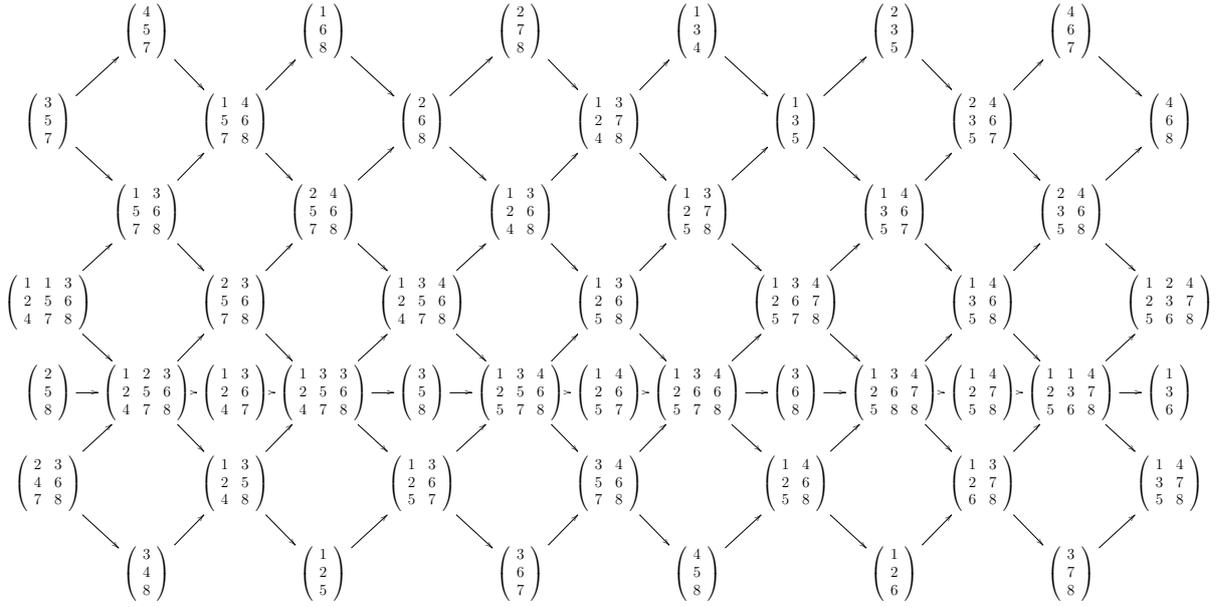
\begin{figure}
$$\scalemath{0.48}{
\xymatrix@-6mm@C-0.2cm{
& & \text{$\left(\begin{array}{c} 4\\ 5\\ 7 \end{array}\right)$} \ar[rdd] & &    \text{$\left(\begin{array}{c} 1\\ 6\\ 8 \end{array}\right)$}    \ar[rdd] & &  \text{$\left(\begin{array}{c} 2\\ 7\\ 8 \end{array}\right)$} \ar[rdd] &&   \text{$\left(\begin{array}{c} 1\\ 3\\ 4 \end{array}\right)$}  \ar[rdd]&&  \text{$\left(\begin{array}{c} 2\\ 3\\ 5 \end{array}\right)$}   \ar[rdd]&& \text{$\left(\begin{array}{c} 4\\ 6\\ 7 \end{array}\right)$} \ar[rdd] \\
\\
& \text{$\left(\begin{array}{c} 3\\ 5\\ 7 \end{array}\right)$} \ar[ruu] \ar[rdd] & & \text{$\left(\begin{array}{cc} 1&4\\ 5&6\\ 7&8 \end{array}\right)$} \ar[ruu] \ar[rdd] & &    \text{$\left(\begin{array}{c} 2\\ 6\\ 8 \end{array}\right)$}   \ar[ruu]\ar[rdd] & &   \text{$\left(\begin{array}{cc} 1&3\\ 2&7\\ 4&8 \end{array}\right)$}  \ar[ruu] \ar[rdd] &&  \text{$\left(\begin{array}{c} 1\\ 3\\ 5 \end{array}\right)$} \ar[ruu]  \ar[rdd]&&  \text{$\left(\begin{array}{cc} 2&4\\ 3&6\\ 5&7 \end{array}\right)$} \ar[ruu]  \ar[rdd]&& \text{$\left(\begin{array}{c} 4\\ 6\\ 8 \end{array}\right)$}    \\
\\
& &  \text{$\left(\begin{array}{cc} 1&3\\ 5&6\\ 7&8 \end{array}\right)$}  \ar[ruu] \ar[rdd] &&   \text{$\left(\begin{array}{cc} 2&4\\ 5&6\\ 7&8 \end{array}\right)$}    \ar[ruu]   \ar[rdd] &&  \text{$\left(\begin{array}{cc} 1&3\\ 2&6\\ 4&8 \end{array}\right)$}   \ar[ruu]  \ar[rdd] &&   \text{$\left(\begin{array}{cc} 1&3\\ 2&7\\ 5&8 \end{array}\right)$}   \ar[ruu]  \ar[rdd]&&  \text{$\left(\begin{array}{cc} 1&4\\ 3&6\\ 5&7 \end{array}\right)$}  \ar[ruu]   \ar[rdd]&& \text{$\left(\begin{array}{cc} 2&4\\ 3&6\\ 5&8 \end{array}\right)$} \ar[ruu]  \ar[rdd] \\
\\
&  \text{$\left(\begin{array}{ccc} 1&1&3\\ 2&5&6\\ 4&7&8 \end{array}\right)$}   \ar[rdd] \ar[ruu] & &  \text{$\left(\begin{array}{cc} 2&3\\ 5&6\\ 7&8 \end{array}\right)$} \ar[rdd]\ar[ruu] & &    \text{$\left(\begin{array}{ccc} 1&3&4\\ 2&5&6\\ 4&7&8 \end{array}\right)$} \ar[rdd] \ar[ruu] & &   \text{$\left(\begin{array}{cc} 1&3\\ 2&6\\ 5&8 \end{array}\right)$}   \ar[ruu]\ar[rdd]&&  \text{$\left(\begin{array}{ccc} 1&3&4\\ 2&6&7\\ 5&7&8 \end{array}\right)$} \ar[ruu]\ar[rdd]&&  \text{$\left(\begin{array}{cc} 1&4\\ 3&6\\ 5&8 \end{array}\right)$} \ar[ruu]\ar[rdd]&&  \text{$\left(\begin{array}{ccc} 1&2&4\\ 2&3&7\\ 5&6&8 \end{array}\right)$}  \\
 \\
& \ar[r]   \text{$\left(\begin{array}{c} 2\\ 5\\ 8 \end{array}\right)$}  &  \ar[ruu]\ar[rdd] \ar[r]  \text{$\left(\begin{array}{ccc} 1&2&3\\ 2&5&6\\ 4&7&8 \end{array}\right)$}    &\ar[r]  \text{$\left(\begin{array}{cc} 1&3\\ 2&6\\ 4&7 \end{array}\right)$} &  \text{$\left(\begin{array}{ccc} 1&3&3\\ 2&5&6\\ 4&7&8 \end{array}\right)$} \ar[r]   \ar[ruu] \ar[rdd]   & \ar[r]  \text{$\left(\begin{array}{c} 3\\ 5\\ 8 \end{array}\right)$}     &  \text{$\left(\begin{array}{ccc} 1&3&4\\ 2&5&6\\ 5&7&8 \end{array}\right)$} \ar[r]  \ar[ruu] \ar[rdd]   & \ar[r]   \text{$\left(\begin{array}{cc} 1&4\\ 2&6\\ 5&7 \end{array}\right)$} & \text{$\left(\begin{array}{ccc} 1&3&4\\ 2&6&6\\ 5&7&8 \end{array}\right)$}  \ar[r]  \ar[rdd]\ar[ruu] & \ar[r]   \text{$\left(\begin{array}{c} 3\\ 6\\ 8 \end{array}\right)$}  &  \ar[r]  \ar[rdd]\ar[ruu] \text{$\left(\begin{array}{ccc} 1&3&4\\ 2&6&7\\ 5&8&8 \end{array}\right)$}   &\ar[r]   \text{$\left(\begin{array}{cc} 1&4\\ 2&7\\ 5&8 \end{array}\right)$} & \text{$\left(\begin{array}{ccc} 1&1&4\\ 2&3&7\\ 5&6&8 \end{array}\right)$} \ar[r]\ar[rdd]\ar[ruu]  & \text{$\left(\begin{array}{c} 1\\ 3\\ 6 \end{array}\right)$}      \\
 \\
& \text{$\left(\begin{array}{cc} 2&3\\ 4&6\\ 7&8 \end{array}\right)$}   \ar[ruu] \ar[rdd]&&   \text{$\left(\begin{array}{cc} 1&3\\ 2&5\\ 4&8 \end{array}\right)$}   \ar[ruu] \ar[rdd]&&  \text{$\left(\begin{array}{cc} 1&3\\ 2&6\\ 5&7 \end{array}\right)$}     \ar[ruu] \ar[rdd]&&   \text{$\left(\begin{array}{cc} 3&4\\ 5&6\\ 7&8 \end{array}\right)$}   \ar[ruu] \ar[rdd]&&   \text{$\left(\begin{array}{cc} 1&4\\ 2&6\\ 5&8 \end{array}\right)$}  \ar[ruu] \ar[rdd]&&  \text{$\left(\begin{array}{cc} 1&3\\ 2&7\\ 6&8 \end{array}\right)$}    \ar[ruu] \ar[rdd]&&   \text{$\left(\begin{array}{cc} 1&4\\ 3&7\\ 5&8 \end{array}\right)$}  \\
 \\
&&  \text{$\left(\begin{array}{c} 3\\ 4\\ 8 \end{array}\right)$} \ar[ruu]&&  \text{$\left(\begin{array}{c} 1\\ 2\\ 5 \end{array}\right)$}   \ar[ruu]&&  \text{$\left(\begin{array}{c} 3\\ 6\\ 7 \end{array}\right)$} \ar[ruu]&&  \text{$\left(\begin{array}{c} 4\\ 5\\ 8 \end{array}\right)$}  \ar[ruu]&&  \text{$\left(\begin{array}{c} 1\\ 2\\ 6 \end{array}\right)$}   \ar[ruu]&&  \text{$\left(\begin{array}{c} 3\\ 7\\ 8 \end{array}\right)$} \ar[ruu] \\
} }
$$
\caption{Auslander--Reiten quiver of $\mathcal{C}_{E_8}$ with vertices labelled by tableaux, Part 1.}
\label{fig: AR quiver E8 part 1 tableaux}
\end{figure}

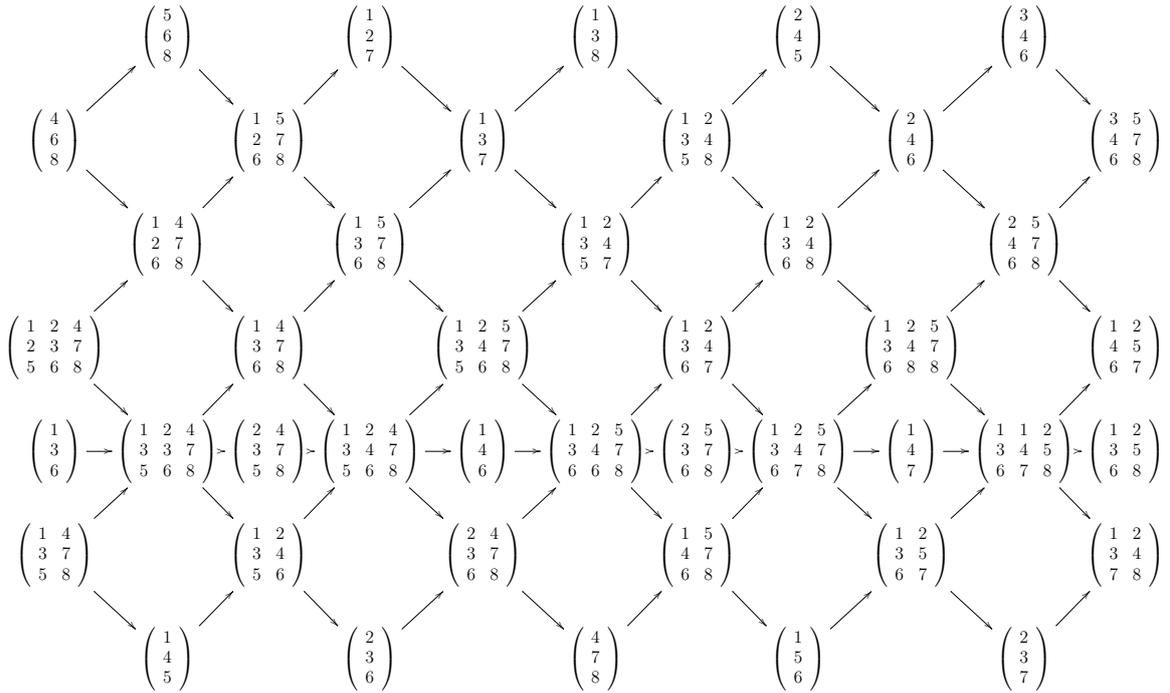
\begin{figure}
$$\scalemath{0.55}{
\xymatrix@-6mm@C-0.2cm{
&&  \text{$\left(\begin{array}{c} 5\\ 6\\ 8 \end{array}\right)$} \ar[rdd] && \text{$\left(\begin{array}{c} 1\\ 2\\ 7 \end{array}\right)$} \ar[rdd]  && \text{$\left(\begin{array}{c} 1\\ 3\\ 8 \end{array}\right)$} \ar[rdd] & & \text{$\left(\begin{array}{c} 2\\ 4\\ 5 \end{array}\right)$}  \ar[rdd] & &   \text{$\left(\begin{array}{c} 3\\ 4\\ 6 \end{array}\right)$}  \ar[rdd]  \\
\\
& \text{$\left(\begin{array}{c} 4\\ 6\\ 8 \end{array}\right)$}  \ar[ruu]  \ar[rdd]&&  \text{$\left(\begin{array}{cc} 1&5\\ 2&7\\ 6&8 \end{array}\right)$} \ar[ruu] \ar[rdd] && \text{$\left(\begin{array}{c} 1\\ 3\\ 7 \end{array}\right)$} \ar[ruu] \ar[rdd] && \text{$\left(\begin{array}{cc} 1&2\\ 3&4\\ 5&8 \end{array}\right)$} \ar[ruu] \ar[rdd] & &    \text{$\left(\begin{array}{c} 2\\ 4\\ 6 \end{array}\right)$}   \ar[ruu]\ar[rdd] &&  \text{$\left(\begin{array}{cc} 3&5\\ 4&7\\ 6&8 \end{array}\right)$}  \\
\\
&&  \text{$\left(\begin{array}{cc} 1&4\\ 2&7\\ 6&8 \end{array}\right)$} \ar[ruu]  \ar[rdd] && \text{$\left(\begin{array}{cc} 1&5\\ 3&7\\ 6&8 \end{array}\right)$} \ar[ruu] \ar[rdd] & & \text{$\left(\begin{array}{cc} 1&2\\ 3&4\\ 5&7 \end{array}\right)$}  \ar[ruu] \ar[rdd] & &   \text{$\left(\begin{array}{cc} 1&2\\ 3&4\\ 6&8 \end{array}\right)$} \ar[ruu]   \ar[rdd] & &   \text{$\left(\begin{array}{cc} 2&5\\ 4&7\\ 6&8 \end{array}\right)$}  \ar[ruu]  \ar[rdd]  \\
\\
&  \text{$\left(\begin{array}{ccc} 1&2&4\\ 2&3&7\\ 5&6&8 \end{array}\right)$} \ar[ruu] \ar[rdd] && \text{$\left(\begin{array}{cc} 1&4\\ 3&7\\ 6&8 \end{array}\right)$} \ar[rdd] \ar[ruu] && \text{$\left(\begin{array}{ccc} 1&2&5\\ 3&4&7\\ 5&6&8 \end{array}\right)$} \ar[ruu] \ar[rdd] &&  \text{$\left(\begin{array}{cc} 1&2\\ 3&4\\ 6&7 \end{array}\right)$} \ar[rdd]\ar[ruu] & &    \text{$\left(\begin{array}{ccc} 1&2&5\\ 3&4&7\\ 6&8&8 \end{array}\right)$} \ar[rdd] \ar[ruu] &&  \text{$\left(\begin{array}{cc} 1&2\\ 4&5\\ 6&7 \end{array}\right)$} \\
 \\
& \text{$\left(\begin{array}{c} 1\\ 3\\ 6 \end{array}\right)$} \ar[r]   & \text{$\left(\begin{array}{ccc} 1&2&4\\ 3&3&7\\ 5&6&8 \end{array}\right)$} \ar[ruu] \ar[rdd] \ar[r] & \text{$\left(\begin{array}{cc} 2&4\\ 3&7\\ 5&8 \end{array}\right)$}  \ar[r] & \text{$\left(\begin{array}{ccc} 1&2&4\\ 3&4&7\\ 5&6&8 \end{array}\right)$}  \ar[r] \ar[ruu] \ar[rdd]  &  \text{$\left(\begin{array}{c} 1\\ 4\\ 6 \end{array}\right)$} \ar[r]  & \text{$\left(\begin{array}{ccc} 1&2&5\\ 3&4&7\\ 6&6&8 \end{array}\right)$}  \ar[r]    \ar[ruu]\ar[rdd] &  \ar[r]    \text{$\left(\begin{array}{cc} 2&5\\ 3&7\\ 6&8 \end{array}\right)$}  & \text{$\left(\begin{array}{ccc} 1&2&5\\ 3&4&7\\ 6&7&8 \end{array}\right)$}  \ar[r] \ar[ruu] \ar[rdd]  &  \ar[r]   \text{$\left(\begin{array}{c} 1\\ 4\\ 7 \end{array}\right)$}     & \text{$\left(\begin{array}{ccc} 1&1&2\\ 3&4&5\\ 6&7&8 \end{array}\right)$}  \ar[r]    \ar[ruu] \ar[rdd] & \text{$\left(\begin{array}{cc} 1&2\\ 3&5\\ 6&8 \end{array}\right)$}  \\
 \\
&   \text{$\left(\begin{array}{cc} 1&4\\ 3&7\\ 5&8 \end{array}\right)$}  \ar[rdd]  \ar[ruu] && \text{$\left(\begin{array}{cc} 1&2\\ 3&4\\ 5&6 \end{array}\right)$}  \ar[ruu] \ar[rdd] && \text{$\left(\begin{array}{cc} 2&4\\ 3&7\\ 6&8 \end{array}\right)$} \ar[ruu] \ar[rdd]  &&   \text{$\left(\begin{array}{cc} 1&5\\ 4&7\\ 6&8 \end{array}\right)$}   \ar[ruu] \ar[rdd]&&  \text{$\left(\begin{array}{cc} 1&2\\ 3&5\\ 6&7 \end{array}\right)$}    \ar[ruu] \ar[rdd]&&  \text{$\left(\begin{array}{cc} 1&2\\ 3&4\\ 7&8 \end{array}\right)$}  \\
 \\
&&  \text{$\left(\begin{array}{c} 1\\ 4\\ 5 \end{array}\right)$} \ar[ruu] && \text{$\left(\begin{array}{c} 2\\ 3\\ 6 \end{array}\right)$} \ar[ruu] &&  \text{$\left(\begin{array}{c} 4\\ 7\\ 8 \end{array}\right)$} \ar[ruu]&&  \text{$\left(\begin{array}{c} 1\\ 5\\ 6 \end{array}\right)$}   \ar[ruu]&&   \text{$\left(\begin{array}{c} 2\\ 3\\ 7 \end{array}\right)$}  \ar[ruu] \\
} }
$$
\caption{Auslander--Reiten quiver of $\mathcal{C}_{E_8}$ with vertices labelled by tableaux, Part 2.}
\label{fig: AR quiver E8 part 2 tableaux}
\end{figure}

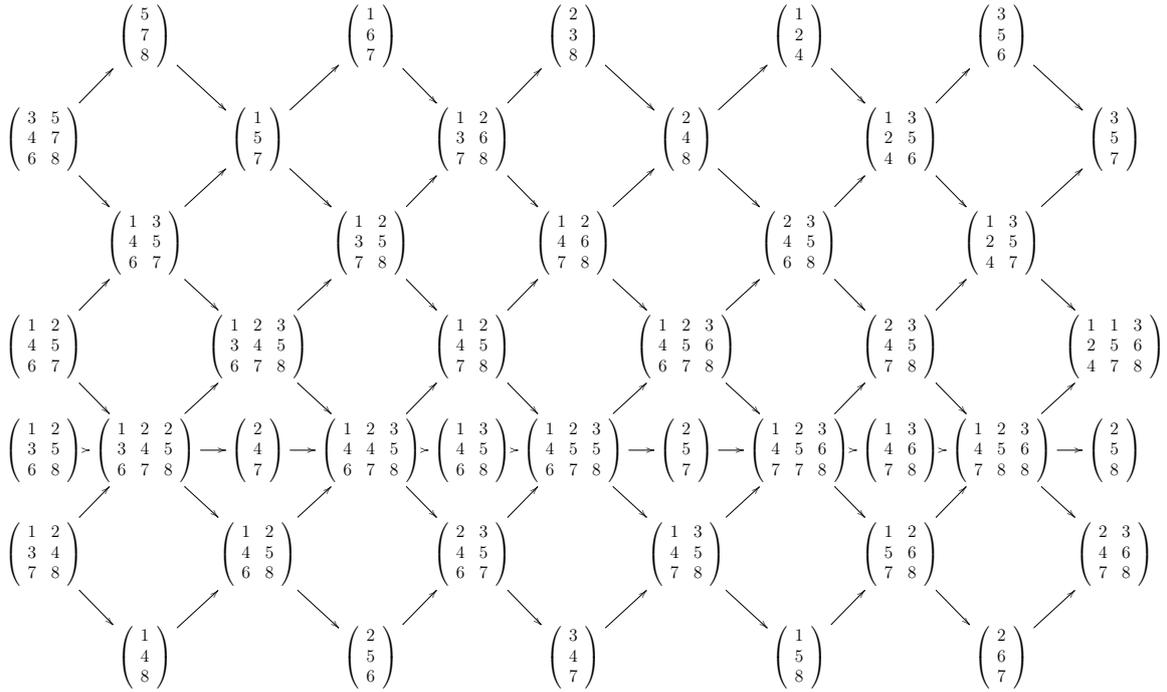
\begin{figure}
$$\scalemath{0.55}{
\xymatrix@-6mm@C-0.2cm{
&&  \text{$\left(\begin{array}{c} 5\\ 7\\ 8 \end{array}\right)$}    \ar[rdd]&&  \text{$\left(\begin{array}{c} 1\\ 6\\ 7 \end{array}\right)$} \ar[rdd]&& \text{$\left(\begin{array}{c} 2\\ 3\\ 8 \end{array}\right)$}  \ar[rdd]&&  \text{$\left(\begin{array}{c} 1\\ 2\\ 4 \end{array}\right)$} \ar[rdd] && \text{$\left(\begin{array}{c} 3\\ 5\\ 6 \end{array}\right)$} \ar[rdd] &&  \\
\\
&   \text{$\left(\begin{array}{cc} 3&5\\ 4&7\\ 6&8 \end{array}\right)$}  \ar[ruu] \ar[rdd] &&  \text{$\left(\begin{array}{c} 1\\ 5\\ 7 \end{array}\right)$}  \ar[ruu]  \ar[rdd]&&  \text{$\left(\begin{array}{cc} 1&2\\ 3&6\\ 7&8 \end{array}\right)$}  \ar[ruu]  \ar[rdd]&& \text{$\left(\begin{array}{c} 2\\ 4\\ 8 \end{array}\right)$}  \ar[ruu]  \ar[rdd]&&  \text{$\left(\begin{array}{cc} 1&3\\ 2&5\\ 4&6 \end{array}\right)$} \ar[ruu] \ar[rdd] && \text{$\left(\begin{array}{c} 3\\ 5\\ 7 \end{array}\right)$}  \\
\\
&&  \text{$\left(\begin{array}{cc} 1&3\\ 4&5\\ 6&7 \end{array}\right)$}   \ar[ruu]  \ar[rdd]&&  \text{$\left(\begin{array}{cc} 1&2\\ 3&5\\ 7&8 \end{array}\right)$}  \ar[ruu]   \ar[rdd]&& \text{$\left(\begin{array}{cc} 1&2\\ 4&6\\ 7&8 \end{array}\right)$}  \ar[ruu]   \ar[rdd]&&  \text{$\left(\begin{array}{cc} 2&3\\ 4&5\\ 6&8 \end{array}\right)$} \ar[ruu]  \ar[rdd] && \text{$\left(\begin{array}{cc} 1&3\\ 2&5\\ 4&7 \end{array}\right)$} \ar[ruu] \ar[rdd]  \\
\\
&   \text{$\left(\begin{array}{cc} 1&2\\ 4&5\\ 6&7 \end{array}\right)$} \ar[ruu]\ar[rdd]&&  \text{$\left(\begin{array}{ccc} 1&2&3\\ 3&4&5\\ 6&7&8 \end{array}\right)$} \ar[ruu]\ar[rdd]&&   \text{$\left(\begin{array}{cc} 1&2\\ 4&5\\ 7&8 \end{array}\right)$}  \ar[ruu]\ar[rdd]&&  \text{$\left(\begin{array}{ccc} 1&2&3\\ 4&5&6\\ 6&7&8 \end{array}\right)$} \ar[ruu] \ar[rdd] && \text{$\left(\begin{array}{cc} 2&3\\ 4&5\\ 7&8 \end{array}\right)$} \ar[rdd] \ar[ruu] && \text{$\left(\begin{array}{ccc} 1&1&3\\ 2&5&6\\ 4&7&8 \end{array}\right)$}   \\
 \\
&  \ar[r]  \text{$\left(\begin{array}{cc} 1&2\\ 3&5\\ 6&8 \end{array}\right)$}   &  \text{$\left(\begin{array}{ccc} 1&2&2\\ 3&4&5\\ 6&7&8 \end{array}\right)$}  \ar[r]     \ar[rdd]\ar[ruu]&  \ar[r] \text{$\left(\begin{array}{c} 2\\ 4\\ 7 \end{array}\right)$}   & \text{$\left(\begin{array}{ccc} 1&2&3\\ 4&4&5\\ 6&7&8 \end{array}\right)$}  \ar[r]     \ar[rdd]\ar[ruu]&  \ar[r]   \text{$\left(\begin{array}{cc} 1&3\\ 4&5\\6&8 \end{array}\right)$}   & \text{$\left(\begin{array}{ccc} 1&2&3\\ 4&5&5\\ 6&7&8 \end{array}\right)$}  \ar[r]   \ar[rdd]\ar[ruu]& \text{$\left(\begin{array}{c} 2\\ 5\\ 7 \end{array}\right)$} \ar[r]  & \text{$\left(\begin{array}{ccc} 1&2&3\\ 4&5&6\\ 7&7&8 \end{array}\right)$}  \ar[r]  \ar[ruu] \ar[rdd] & \text{$\left(\begin{array}{cc} 1&3\\ 4&6\\ 7&8 \end{array}\right)$} \ar[r] & \text{$\left(\begin{array}{ccc} 1&2&3\\ 4&5&6\\ 7&8&8 \end{array}\right)$}  \ar[ruu] \ar[rdd] \ar[r] & \text{$\left(\begin{array}{c} 2\\ 5\\ 8 \end{array}\right)$}  \\
 \\
&   \text{$\left(\begin{array}{cc} 1&2\\ 3&4\\ 7&8 \end{array}\right)$}  \ar[ruu] \ar[rdd]&&  \text{$\left(\begin{array}{cc} 1&2\\ 4&5\\ 6&8 \end{array}\right)$}  \ar[ruu] \ar[rdd]&&  \text{$\left(\begin{array}{cc} 2&3\\ 4&5\\ 6&7 \end{array}\right)$}    \ar[ruu] \ar[rdd]&&   \text{$\left(\begin{array}{cc} 1&3\\ 4&5\\ 7&8 \end{array}\right)$}  \ar[rdd]  \ar[ruu] && \text{$\left(\begin{array}{cc} 1&2\\ 5&6\\ 7&8 \end{array}\right)$}  \ar[ruu] \ar[rdd] && \text{$\left(\begin{array}{cc} 2&3\\ 4&6\\ 7&8 \end{array}\right)$}  \\
 \\
&&  \text{$\left(\begin{array}{c} 1\\ 4\\ 8 \end{array}\right)$}   \ar[ruu]&& \text{$\left(\begin{array}{c} 2\\ 5\\ 6 \end{array}\right)$} \ar[ruu]&&   \text{$\left(\begin{array}{c} 3\\ 4\\ 7 \end{array}\right)$} \ar[ruu]&&  \text{$\left(\begin{array}{c} 1\\ 5\\ 8 \end{array}\right)$} \ar[ruu] && \text{$\left(\begin{array}{c} 2\\ 6\\ 7 \end{array}\right)$} \ar[ruu] \\
} }
$$
\caption{Auslander--Reiten quiver of $\mathcal{C}_{E_8}$ with vertices labelled by tableaux, Part 3.}
\label{fig: AR quiver E8 part 3 tableaux}
\end{figure}


\clearpage

\begin{thebibliography}{10}

\bibitem{BFZ96}
A. Berenstein, S. Fomin, and A. Zelevinsky, 
\newblock Parametrizations of canonical bases and totally positive matrices. \newblock {\em Adv.~Math.} \textbf{122} (1996), no. 1, 49--149.

\bibitem{BZ}
A. Berenstein and A. Zelevinsky, 
\newblock Total positivity in Schubert varieties. 
\newblock {\em Comm.~Math.~Helv.} \textbf{72} (1997), no. 1, 128--166.

\bibitem{BM}
K. Baur and R. J. Marsh.
\newblock A geometric description of $m$-cluster categories.
\newblock {\em Trans.~Amer.~Math.~Soc.} \textbf{360} (2008), no. 11, 5789--5803.

\bibitem{BMd}
K. Baur and R. J. Marsh.
\newblock  A Geometric Description of the $m$-cluster Categories of Type $D_n$,
\newblock {\em Int.~Math.~Res.~Not.} (2007), Volume 2007, 1073--7928.

\bibitem{BMRRT}
A. B. Buan, R.  Marsh, M. Reineke, I. Reiten, and G. Todorov,
\newblock Tilting theory and cluster combinatorics.
\newblock {\em Adv.~Math.} \textbf{204} (2006), no. 2, 572--618.

\bibitem{CDFL}
W. Chang, B. Duan, C. Fraser, and J.-R. Li,
\newblock Quantum affine algebras and Grassmannians.
\newblock {\em arXiv:1907.13575}, 2019.

\bibitem{DFG}
J. Drummond, J. Foster and \"{O}. G\"{u}urd\u{o}gan, 
\newblock Cluster Adjacency Properties of Scattering Amplitudes in $N=4$ Supersymmetric Yang-Mills Theory. 
\newblock {\em Phys.~Rev.~Lett.} \textbf{120}, no. 16, 161601 (2018).


\bibitem{FR}
S. Fomin and N. Reading.
\newblock Generalized cluster complexes and Coxeter combinatorics.
\newblock {\em Int.~Math.~Res.~Not.} (2005), no. 44, 2709--2757.

\bibitem{FZ}
S. Fomin and A. Zelevinsky.
\newblock Cluster algebras. I. Foundations. 
\newblock {\em J.~Amer.~Math.~Soc.} \textbf{15} (2002), no. 2, 497--529.

\bibitem{FZy}
S. Fomin and A. Zelevinsky.
\newblock $Y$-systems and generalized associahedra.
\newblock {\em Ann.~of Math. (2)} \textbf{158} (2003), no. 3, 977--1018.

\bibitem{GLS}
C. Geiss, B. Leclerc, and J. Schr\"{o}er, 
\newblock Generic bases for cluster algebras and the Chamber ansatz. 
\newblock {\em J.~Am.~Math.~Soc.} \textbf{25} (2012), no. 1, 21--76.

\bibitem{Happel2}
D. Happel,
\newblock Triangulated categories in the representation theory of finite-dimensional algebras. 
\newblock {\em Lond.~Math.~Soc.~Lecture Note Series}, (1988), \textbf{119}, Cambridge University Press, Cambridge.

\bibitem{JKS} 
B. T. Jensen, A. D. King, X. P. Su,
\newblock A categorification of Grassmannian cluster algebras. 
\newblock {\em Proc.~Lond.~Math.~Soc. (3)} (2016), \textbf{113} no. 2, 185--212.

\bibitem{Kell08}
B. Keller,
\newblock Calabi-Yau triangulated categories.
\newblock In {\em Trends in representation theory of algebras and related
topics}, (2008), EMS Ser.~Congr.~Rep., pages 467--489. Eur. Math. Soc., Z\"{u}rich.

\bibitem{K2}
B. Keller,
\newblock Cluster algebras, quiver representations and triangulated categories.
\newblock In T. Holm, P. J{\o}rgensen, and R. Rouquier (Eds.), Triangulated categories, 76--160, 
\newblock {\em London Math.~Soc.~Lecture Note Ser.}
 \textbf{375}, Cambridge Univ. Press, Cambridge (2010).
  
\bibitem{KL} 
D. Kazhdan and G. Lusztig, 
\newblock Representations of Coxeter groups and Hecke algebras. 
\newblock {\em Invent.~Math.} \text{53} (1979), no. 2, 165--184.

\bibitem{L3}
L. Lamberti.
\newblock Combinatorial model for the cluster categories of type {E}.
\newblock {\em J.~Algebraic Combin.} \textbf{41} (2015), no. 4, 1023--1054.

\bibitem{LPSV}
T. \L{}ukowski, M. Parisi, M. Spradlin, and A. Volovich,
\newblock Cluster adjacency for $m=2$ yangian invariants.
\newblock {\em arXiv:1908.07618}, (2019).

\bibitem{MS}
R. Marsh and J. S. Scott, 
\newblock Twists of Pl\"{u}cker coordinates as dimer partition functions.
\newblock {\em Comm.~Math.~Phys.} \textbf{341} (2016), no. 3, 821--884.

\bibitem{MSSV}
J. Mago, A. Schreiber, M. Spradlin, and A. Volovich,
\newblock Yangian invariants and cluster adjacency in $n=4$ Yang-Mills.
\newblock {\em arXiv:1906.10682}, (2019).

\bibitem{Miyachi}
J.-i. Miyachi and A. Yekutieli,
\newblock Derived Picard groups of finite-dimensional hereditary algebras.
\newblock {\em Compositio Math.} \textbf{129} (2001), no. 3, 341--368.

\bibitem{Riedtmann}
C. Riedtmann,
\newblock Algebren, {D}arstellungsk\"{o}cher, \"{U}berlagerungen und zur\"{u}ck.
\newblock {\em Comment.~Math.~Helv.} \textbf{55} (1980), no. 2, 199--224.

\bibitem{Sco}
J. S. Scott,
\newblock Grassmannians and cluster algebras.
\newblock {\em Proc.~London Math.~Soc.} (3) \textbf{92} (2006), no. 2, 345--380.

\bibitem{Ses}
C. S. Seshadri, 
\newblock Introduction to the theory of standard monomials, Second edition. \newblock {\em Texts and Readings in Mathematics} \textbf{46}, Hindustan Book Agency, New Delhi, (2014).

\bibitem{Th07}
H. Thomas,
\newblock Defining an {$m$}-cluster category.
\newblock {\em J.~Algebra} \textbf{318} (2007), no. 1, 37--46.

\bibitem{Tz06}
E. Tzanaki,
\newblock Polygon dissections and some generalizations of cluster complexes.
\newblock {\em J.~Combin.~Theory Ser.~A} \textbf{113} (2006), no. 6, 1189--1198.

\bibitem{Zhu06}
B. Zhu,
\newblock Equivalences between cluster categories.
\newblock {\em J.~Algebra} \textbf{304} (2006), no. 2, 832--850.

\bibitem{Zhu08}
B. Zhu,
\newblock Generalized cluster complexes via quiver representations.
\newblock {\em J.~Algebraic Combin.} \textbf{27} (2008), no. 1, 35--54.

\bibitem{Zhu11}
B. Zhu,
\newblock Cluster-tilted algebras and their intermediate coverings.
\newblock {\em Comm.~Algebra}, \textbf{39} (2011), no. 7.
2437--2448.

\end{thebibliography}
\end{document}